\documentclass[submission,copyright,creativecommons]{eptcs}
\providecommand{\event}{AiML 2026} 

\usepackage{aiml26}

\usepackage{iftex}

\ifpdf
  \usepackage{underscore}         
  \usepackage[T1]{fontenc}        
\else
  \usepackage{breakurl}           
\fi

\usepackage{amsmath,amssymb}

\usepackage{tikz}
\usetikzlibrary{shapes.geometric,calc,decorations.markings,hobby}

\newcommand{\defeq}{:=}
\newcommand{\Dmd}{\Diamond}

\newcommand{\un}[1]{\underbar{#1}}
\newcommand{\E}[1]{\exists #1\,}
\newcommand{\A}[1]{\forall #1\,}

\title{On Interpretations of Normal Modal Logics}
\author{Lev V.~Dvorkin
\thanks{The author is a winner of the “Junior Leader” competition grant held by the
“BASIS” Foundation for the Development of Theoretical Physics and Mathematics.}
\institute{Lomonosov Moscow State University, Russia}
\email{lev\_135@mail.ru}
}

\newcommand{\titlerunning}{On Interpretations of Normal Modal Logics}
\newcommand{\authorrunning}{L.V.~Dvorkin}

\hypersetup{
  bookmarksnumbered,
  pdftitle    = {\titlerunning},
  pdfauthor   = {\authorrunning},
  pdfsubject  = {EPTCS},               
  pdfkeywords = {keyword1, keyword2} 
}

\begin{document}
\maketitle

\begin{abstract}
  We study interpretations of modal logics in one another where the Boolean connectives are
  interpreted identically and the modal operator diamond is interpreted by an arbitrary formula
  $\alpha(p)$. Clearly, such a formula $\alpha(p)$ defines an interpretation of a normal modal logic
  whenever $\alpha(p)$ is additive (that is, preserves disjunction) and normal (that is, preserves
  bottom) in the target logic. In the present paper, we provide a complete description of all
  additive and normal formulas in five prominent modal logics: K, GL, Grz, S4, and S5. For K, GL,
  and S5, we also describe all additive and normal formulas with parameters.
\end{abstract}

\section{Introduction}
  \label{s:intro}
  Let $\mathcal{V}=\{p_0,p_1,\dots\}$ be a countable set of \emph{propositional variables},  
  $\mathrm{Fm}_\mathcal{L}$ and $\mathrm{Fm}_{\mathcal{L}_0}$ be the sets of all formulas in
  propositional languages $\mathcal{L}$ and $\mathcal{L}_0$ with variables from $\mathcal{V}$. 
  A \emph{translation} from $\mathcal{L}$ into $\mathcal{L}_0$ is an arbitrary mapping
  $\tau : \mathrm{Fm}_\mathcal{L} \to \mathrm{Fm}_{\mathcal{L}_0}$. 
  If logics $\Lambda \subseteq \mathrm{Fm}_\mathcal{L}$ and
  $\Lambda_0 \subseteq \mathrm{Fm}_{\mathcal{L}_0}$ are such that
  $\Lambda=\tau^{-1}\Lambda_0$
  (i.e., $\Lambda \vdash \varphi \Leftrightarrow \Lambda_0 \vdash \tau\varphi$), then we say that $\tau$
  \emph{defines an (exact) interpretation of $\Lambda$ in $\Lambda_0$}. There are several well-known
  examples of such interpretations:
  \begin{itemize}
    \item Glivenko's interpretation~\cite{Gliv29} of classical logic in intuitionistic logic:
          \begin{gather*}
            {\rm Cl} \vdash \varphi \Leftrightarrow {\rm Int} \vdash \neg\neg\varphi
          \end{gather*}
          and its modal analog~\cite{Mats55}:
          $\mathrm{S5} \vdash \varphi \Leftrightarrow \mathrm{S4} \vdash \Box\Dmd\varphi$.
    \item G\"odel's interpretation~\cite{God33} of intuitionistic logic in $\mathrm{S4}$:
          \begin{gather*}
            {\rm Int} \vdash \varphi \Leftrightarrow \mathrm{S4} \vdash \tau_G \varphi,
          \end{gather*}
          where $\tau_G$ adds $\Box$ on each subformula.
    \item Interpretations of $\rm KT$ in $\rm K$~\cite{Pell84}, of $\mathrm{S4}$ in $\mathrm{K4}$~\cite{Meschi78}, 
          and of $\mathrm{Grz}$ in $\mathrm{GL}$~\cite{Gold78,KuznMur80}:
          \begin{align*}
            \mathrm{KT} \vdash \varphi &\Leftrightarrow \mathrm{K} \vdash \tau_{\rm refl}
            \varphi,\\
            \mathrm{S4} \vdash \varphi &\Leftrightarrow \mathrm{K4} \vdash \tau_{\rm refl}
            \varphi,\\
            \mathrm{Grz} \vdash \varphi &\Leftrightarrow \mathrm{GL} \vdash \tau_{\rm refl}
            \varphi,
          \end{align*}
          where $\tau_{\rm refl}$ preserves variables, commutes with all Boolean connectives, and
          maps formulas of the form $\Dmd\psi$ to $\Dmd\tau_{\rm refl}\psi\lor\tau_{\rm refl}\psi$.
    \item The interpretation of Solovay's logic in $\mathrm{GL}$~\cite{Sol76}:
          \begin{gather*}
            {\rm S} \vdash \varphi \Leftrightarrow \mathrm{GL} \vdash \bigwedge_{\Dmd\psi \in \mathrm{Sub}(\varphi)}(\Box\psi \to \psi) \to \varphi,
          \end{gather*}
          where ${\rm Sub}(\varphi)$ is the set of all subformulas of $\varphi$.
  \end{itemize}
  Let us consider interpretations of modal logics in each other given by
  translations of the following form: for each modality $\Dmd_i$ in $\mathcal{L}$, we fix some
  formula $\alpha_i(p) \in \mathrm{Fm}_{\mathcal{L}_0}$. Then
  $\tau_{\vec \alpha} : \mathrm{Fm}_\mathcal{L} \to \mathrm{Fm}_{\mathcal{L}_0}$ is the translation
  that maps $p_k$ to $p_{2k}$\footnote{We use this technical trick to reserve odd-indexed variables for parameters (see below).}, 
  commutes with Boolean connectives, and sends $\Dmd_i\psi$ to
  $\alpha_i(\tau_{\vec \alpha}\psi)$. Such translations are called \emph{modal-to-modal}
  translations in \cite{French10}. This notion is quite restrictive. For instance, translations from
  $\mathrm{S5}$ to $\mathrm{S4}$ and from $\rm S$ to $\mathrm{GL}$ mentioned above are not of this form 
  (and cannot be replaced by translations of this form, see~\cite[Theorem~5.10]{Zol00}).
  Strictly speaking, $\tau_{\rm refl}$ is also not of this form, but it can be replaced by $\tau_{\Dmd p\lor p}$
  equally well. 
  There are other interesting examples of interpretations given by modal-to-modal
  translations:
  \begin{itemize}
    \item Kracht and Wolter~\cite{KW99} showed that every monotone monomodal logic $\Lambda$ is
          interpretable in a normal bimodal logic via the translation $\tau_{\Box_1\Dmd_2p}$.
    \item Thomason's translation~\cite{Thom75} provides an interpretation of normal polymodal logics in
          monomodal ones. It is not a modal-to-modal translation but has the
          form $\tau_T \varphi \defeq \Dmd\Box\bot \to \tau_{\vec \beta}\varphi$, where
          $\tau_{\vec \beta}$ is a particular modal-to-modal translation.
    \item The logics that are interpreted in monomodal logics by the translation
          $\tau_{\Dmd p \wedge \Dmd\neg p}$ are known as \emph{non-contingency logics}~\cite{MR66,Humb95,Zol02}.
    \item Translation $\tau_{\Box\Diamond p}$ defines an interpretation of the minimal monotone logic ${\rm EM}$
          in $\rm K$~\cite[Theorem~6.1.11]{French10}. 
          Some other interpretations (including interpretations of normal logics in each other) 
          via this translation are discussed in~\cite[Section 2]{Hum06}.  
    \item The provability logic of Niebergall arithmetic w.r.t.\ Peano arithmetic, 
          investigated in~\cite{Dvo24}, is a normal logic, which is interpreted in the polymodal 
          provability logic $\rm GLP$ by $\tau_{\Dmd_1p\lor\Dmd_0(\Dmd_1\top\land p)}$.
  \end{itemize}
  In this paper, we consider only \emph{congruential modal logics}, i.e., the sets of formulas that
  contain classical tautologies and are closed under the rules of modus ponens, substitution, and 
  \emph{equivalent replacement} $\frac{\varphi\leftrightarrow\psi}{\Diamond\varphi\leftrightarrow\Diamond\psi}$.
  It is easy to see that if $\Lambda$ is a modal logic in this sense, then so is 
  the set $\tau_{\vec\alpha}^{-1}\Lambda$ for any modal-to-modal translation $\tau_{\vec\alpha}$. 
  Hence, every modal-to-modal translation defines an interpretation of some logic in a given one.
  In~\cite{Zol00}, it was shown that, for every logic $\Lambda$ such that
  $\mathrm{K} \subseteq \Lambda \subseteq \mathrm{GL}$ or
  $\mathrm{K} \subseteq \Lambda \subseteq \mathrm{Grz}$, infinitely many
  monomodal logics are interpretable in $\Lambda$ by modal-to-modal translations.
  What happens when we restrict our attention to normal logics?
  It is well known that a congruential logic is normal iff it contains the axioms 
  $\Diamond(p\lor q)\leftrightarrow\Diamond p\lor\Diamond q$ and $\Diamond\bot\leftrightarrow\bot$.
  Therefore, $\tau_\alpha$ defines an interpretation of a normal
  logic in $\Lambda$ iff $\alpha$ is \emph{additive} and
  \emph{normal} in $\Lambda$, i.e., the following equivalences hold in $\Lambda$:
  \begin{gather*}
    \alpha(p \vee q) \leftrightarrow \alpha(p) \vee \alpha(q)
    \quad\text{and}\quad
    \alpha(\bot) \leftrightarrow \bot .
  \end{gather*}
  Thus, to understand which normal logics are interpretable in $\Lambda$,
  it is helpful to describe all formulas that are additive and normal in $\Lambda$.
  Such a description for $\mathrm{K}$ was obtained in~\cite{Benth98,Benth26} using
  model-theoretic methods.
  In the present paper, we reprove this result using a different, more
  constructive approach and extend the analysis to four other central modal
  logics: $\mathrm{GL}$, $\mathrm{S5}$, $\mathrm{S4}$, and $\mathrm{Grz}$.
  In particular, we show that there are exactly five pairwise non-equivalent normal 
  additive formulas in $\mathrm{S4}$ and in $\mathrm{Grz}$.

  One can also  consider interpretations with parameters. 
  A classic example is the interpretation of minimal logic in intuitionistic logic~\cite{PraMal68} 
  via the following translation with one parameter $p_1$:
  \begin{gather*}
      \tau\bot := p_1,\quad\tau p_k := p_{2k}\text{\;\;for }k\in\omega,\\
      \tau(\varphi\mathbin{\circ}\psi):=\tau\varphi\mathbin{\circ}\tau\psi\text{\;\;for }\circ\in\{\lor,\land,\to\}.
  \end{gather*}
  Note that, under the general definition of interpretation, there is no distinction between parameters 
  and other variables. 
  Differences only emerge when we consider interpretations of a certain type --- for example, modal-to-modal ones.
  Recall that $\tau_{\vec\alpha}$ maps all variables to variables with even indices.
  Variables $p_{2k+1},k\in\omega$ will be treated as parameters. 
  Parametric modal-to-modal translations are defined as follows: for each modality $\Diamond_i$, 
  we fix a formula $\alpha_i(p, \vec r)$, where $\vec r=\langle r_1,\dots,r_n\rangle$ is a tuple of parameters. 
  Then the translation $\tau_{\vec\alpha}$ 
  maps $p_k$ to $p_{2k}$,
  commutes with Boolean connectives, and
  sends $\Dmd_i\psi$ to $\alpha_i(\tau_{\vec\alpha}\psi, \vec r)$.
  It is easy to verify that, for every logic $\Lambda$, $\tau_{\vec\alpha}^{-1}\Lambda$ is a logic, 
  whence every modal-to-modal translation with parameters defines an interpretation
  of some logic in a given logic.
  For instance, weak Grzegorczyk logic $\rm w\mathrm{Grz}$ is interpreted in $\mathrm{GL}$ via the
  parametric translation $\tau_{\Dmd p \vee (r \wedge p)}$ (see Example~\ref{e:wGrz-in-GL}). 
  Moreover, the use of parameters allows us to interpret $\rm K$, $\rm K4$, and $\rm GL$
  in $\rm KT$, $\rm S4$, and $\rm Grz$ respectively (see Example~\ref{e:K-in-KT}).
  We will provide a characterization of
  additive formulas with parameters for the logics $\mathrm{K}$, $\mathrm{GL}$, and $\mathrm{S5}$.
  The structure of additive formulas in $\mathrm{S4}$ and $\mathrm{Grz}$ is much more complicated,
  so the parametric case for these logics is postponed to a subsequent paper.

  The syntactic characterization of additive formulas can also be studied independently of interpretations, 
  simply as the problem of characterizing formulas with a certain semantic property. 
  The most well-known results of this type are the \L{}os-Tarski and Lyndon preservation
  theorems from classical model theory.
  Modal analogs of these results were considered in~\cite{dR93,dR*,Benth26,Dvo26}.
  Additive operators in Boolean algebras 
  trace back to the classical work of J\'onsson and Tarski~\cite{JT51}.
  The results of the present paper provide a syntactic characterization of the modal 
  formulas that define such operators.

  The structure of the paper is as follows. 
  In Section~\ref{s:prelims}, we recall some basic notions from modal logic and introduce the main
  objects for the present paper: additive and normal formulas and operators.
  In Section~\ref{s:gen}, we provide some basic results on additive formulas.
  In Section~\ref{s:S5}, we characterize all additive formulas with parameters in $\mathrm{S5}$. 
  This is the simplest, yet non-trivial, case.
  In Section~\ref{s:K-GL}, we do the same for $\mathrm{K}$ and $\mathrm{GL}$. 
  In Section~\ref{s:S4-Grz}, we characterize additive formulas without parameters in $\mathrm{S4}$ and $\mathrm{Grz}$.
  In Section~\ref{s:interp}, we make some remarks regarding the interpretations of normal logics, 
  partly building on the results of the previous sections.
  In Section~\ref{s:prob}, we state some open questions for future research.

\section{Preliminaries}
    \label{s:prelims}
  \subsection{Modal formulas}
    We fix a countable set $\mathcal{V}=\{p_0,p_1,\dots\}$ of \emph{propositional variables}, $p := p_0$, and $q := p_2$.
    Variables with odd indices are called \emph{parameters}.
    The letter $r$ (possibly with indices) will be used as a metavariable ranging over parameters.
    \emph{Modal formulas} are built from variables $p_k$, their
    negations $\neg p_k$, the constants $\bot$ and $\top$ using $\wedge$, $\vee$, $\Dmd$, and $\Box$.
    The negation of a formula is defined recursively by the duality laws. The Boolean connectives $\to$ and
    $\leftrightarrow$ are treated as abbreviations. 
    The set of all modal formulas is denoted by $\rm Fm$.
    For a tuple of parameters $\vec r=\langle r_j\rangle_{j<n}$, 
    let ${\rm Fm}(p,\vec r)$ denote the set of formulas that contain no variables other than $p,r_0,\dots, r_{n-1}$.
    A formula $\varphi$ is \emph{$p$-positive} if it has no occurrence of $\neg p$.
    The set of all formulas from ${\rm Fm}(p,\vec r)$ that are $p$-positive is denoted by ${\rm Fm}^+(p,\vec r)$.
    We will also use the notation ${\rm Fm}(p)$, ${\rm Fm}^+(p)$, and ${\rm Fm}(\vec r)$ with the obvious meanings.
    Note that ${\rm Fm}(\langle\rangle)\subset{\rm Fm}(\vec r) \subset {\rm Fm}^+(p,\vec r) \subset \mathrm{Fm}(p,\vec r)$,
    where ${\rm Fm}(\langle\rangle)$ is the set of all variable-free formulas.
    %
    For $\alpha=\alpha(p,\vec r)\in{\rm Fm}(p,\vec r)$, we put $\alpha_\bot:=\alpha(\bot,\vec r)\in{\rm Fm}(\vec r)$.
  \subsection{Normal modal logic and Kripke semantics}
    Normal modal logics, Kripke frames and models, and p-morphisms are defined as usual (see, e.g.,~\cite{ChZa97}). 
    We also use standard notation for some particular logics: $\mathrm{K}$, $\mathrm{GL}$, $\mathrm{S4}$, $\mathrm{S5}$,
    $\mathrm{Grz}$, etc. For a frame $\mathcal{F}$ (class of frames $\mathcal{C}$) we denote by
    $\mathrm{Log}\, \mathcal{F}$ ($\mathrm{Log}\, \mathcal{C}$) its logic (i.e., the set of all
    formulas that are valid in it). For a logic $\Lambda$, we say that formulas $\varphi$ and $\psi$
    are \emph{$\Lambda$-equivalent} if $\Lambda \vdash \varphi \leftrightarrow \psi$ and denote this
    by $\varphi \sim_\Lambda \psi$.

    A \emph{general frame} is a pair $\mathcal{G}=(\mathcal{F}, P)$, where $\mathcal{F} = (W, R)$ 
    is a Kripke frame and $P\subseteq\mathcal{P}(W)$ is closed under Boolean operations and the \emph{full 
    preimage} operator $R^{-1}:U\mapsto\{w\in W\mid\exists{u\in U}\,w\mathrel{R} u\}$. 
    The \emph{logic of $\mathcal{G}$} is the set $\mathrm{Log}\,\mathcal{G}$
    of all formulas that are true in all models $(\mathcal F, \vartheta)$, where $\vartheta(p_k)\in P$ for all $k\in\omega$.
    Clearly, $\mathrm{Log}\,\mathcal{F}\subseteq\mathrm{Log}\,\mathcal{G}$.
    For a class of general frames $\mathcal{D}$, 
    $\mathrm{Log}\,\mathcal{D}:=\bigcap_{\mathcal G\in\mathcal D}\mathrm{Log}\mathcal G$.
  \subsection{Monotone, additive, and normal operators and formulas}
    Let $\mathcal{F} = (W, R)$ be a Kripke frame.
    For a formula $\varphi\in{\rm Fm}(p, \vec r)$, $\vec r = \langle r_j\rangle_{j<n}$, consider the operator
    \begin{align*}
      \alpha_\mathcal{F} : \mathcal{P}(W)^{1+n} &\to \mathcal{P}(W)\\
      (U, \vec V) &\mapsto \vartheta(\varphi), \text{ where }\vartheta(p) \defeq U, \vartheta(\vec r)
      \defeq \vec V.
    \end{align*}
    Notice that ${\cdot}_\mathcal{F}$ is a homomorphism from the algebra of formulas with $1+n$
    variables into the algebra of $(1+n)$-place operators on $\mathcal{P}(W)$. It is easy to see
    that, for any two formulas $\varphi(p, \vec r)$ and $\psi(p, \vec r)$, 
    \[
        \mathcal{F} \vDash \varphi
        \leftrightarrow \psi \Leftrightarrow \varphi_\mathcal{F} = \psi_\mathcal{F}.
    \]

    An operator $f : \mathcal{P}(W) \to \mathcal{P}(W)$ is
    \begin{itemize}
      \item \emph{monotone} if $U \subseteq V \Rightarrow fU \subseteq fV$ for all $U,V\subseteq W$;
      \item \emph{additive} if $f(U \cup V) = fU \cup fV$ for all $U,V\subseteq W$;
      \item \emph{normal} if $f\emptyset = \emptyset$;
      \item \emph{completely additive} if $f\bigcup_{i\in I}U_i=\bigcup_{i\in I}fU_i$ for every collection of sets
      $\{U_i\}_{i\in I}\subseteq \mathcal{P}(W)$.
    \end{itemize}
    A formula $\varphi(p, \vec r)$ is
    \begin{itemize}
      \item \emph{monotone (w.r.t.\ $p$) in $\Lambda$} if
            $\Lambda \vdash \varphi(p, \vec r) \to \varphi(p \vee q, \vec r)$;
      \item \emph{additive (w.r.t.\ $p$) in $\Lambda$} if $\Lambda \vdash \varphi(p \vee q, \vec r)
            \leftrightarrow \varphi(p, \vec r) \vee \varphi(q, \vec r)$;
      \item \emph{normal (w.r.t.\ $p$) in $\Lambda$} if
            $\Lambda \vdash \varphi(\bot, \vec r) \leftrightarrow \bot$.
    \end{itemize}
    Notice that additivity implies monotonicity in both cases, 
    and all complete additive operators are additive and normal\footnote{
    We assume that the union of the empty collection of sets equals $\emptyset$.}.
    Formulas that are additive in $\Lambda$  will also be called \emph{$\Lambda$-additive}. $\Lambda$-additive
    formulas that are normal in $\Lambda$ will be called \emph{normal $\Lambda$-additive} formulas.

    For a class of frames $\mathcal{C}$, we say that a formula $\varphi(p, \vec r)$ is
    \emph{monotone (additive, normal, completely additive) in $\mathcal{C}$} if, 
    for all $\mathcal{F} = (W, R) \in \mathcal{C}$ and $\vec V \in \mathcal{P}(W)^n$, the operator
    \begin{align*}
        \varphi_\mathcal{F}(\cdot, \vec V) : \mathcal{P}(W) &\to \mathcal{P}(W)\\ 
        U &\mapsto \varphi_\mathcal{F}(U, \vec V)
    \end{align*}
    is monotone (additive, normal, completely additive).
    The following can be easily checked:
    \begin{proposition}
        \label{p:add-c}
      Let $\Lambda = \mathrm{Log}\, \mathcal{C}$. 
      Then  $\varphi(p, \vec r)$ is monotone (additive, normal) in $\Lambda$ iff 
      it is monotone (additive, normal) in $\mathcal{C}$.
    \end{proposition}
    Since our language is finitary, we have no syntactic counterpart to complete additivity.
    One can show (see~\cite[Proposition~2.3]{Benth26}) that a formula is completely additive in the class of all Kripke frames
    iff it is additive in this class, i.e., iff it is ${\rm K}$-additive. 
    It is easy to adapt this argument to the classes of frames defined by some of the usual first-order 
    conditions such as reflexivity, symmetry, and seriality.
    At the same time, the situation is quite different for the classes of transitive frames:
    \begin{proposition}
        \label{p:s4-dbd}
        The formula $\alpha=\Diamond\Box\Diamond p$ is normal ${\rm S4}$-additive, 
        but it is not completely additive in the class of all reflexive transitive frames. 
    \end{proposition}
    \begin{proof}
        One can check that $\alpha(p\lor q)\leftrightarrow\alpha(p)\lor\alpha(q)$ and $\alpha(\bot)\leftrightarrow\bot$
        are derivable in ${\rm S4}$. Here we give a semantic argument which is interesting in its own right.
        We say that a world $v$ of an ${\rm S4}$-frame $\mathcal{F}=(W,R)$ is \emph{maximal} 
        if $v\mathrel{R}u\Rightarrow u\mathrel{R} v$ for all $u\in W$.
        Let $R'$ be the relation on $W$ such that $w\mathrel{R'}v$ iff $w\mathrel{R}v$ and $v$ is maximal.
        It is easy to check that, $\alpha_\mathcal{F}=(R')^{-1}$ whenever $\mathcal{F}$ is finite, 
        whence $\alpha$ is (completely) additive in the class of finite $\rm S4$-frames.
        By Proposition~\ref{p:add-c}, $\alpha$ is ${\rm S4}$-additive. 

        Now let us consider the frame $\mathcal{F}=(\omega, \leq)$. It is easy to see that
        $\alpha_\mathcal{F}\omega=\omega$ and $\alpha_\mathcal{F}\{n\}=\emptyset$ for each $n\in\omega$.
        Therefore, $\alpha_\mathcal{F}$ is not completely additive, and $\alpha$ is not completely additive 
        in the class of all reflexive transitive frames.
    \end{proof}
    \begin{corollary}
        \label{c:k4-dbd}
        The formula $\beta=\tau_{\Dmd p\lor p}\alpha$ is normal ${\rm K4}$-additive,
        but it is not completely additive in the class of all transitive frames.
    \end{corollary}
    \begin{proof}
        Notice that ${\rm S4}=\tau^{-1}_{\Dmd p\lor p}{\rm K4}$ and $\beta_\mathcal{F}=\alpha_\mathcal{F}$ 
        for a reflexive frame $\mathcal{F}$. Then the claim follows from Proposition~\ref{p:s4-dbd}.
    \end{proof}
    We now turn to monotone formulas.
    It is easy to see that every $p$-positive formula is monotone in any normal modal logic. 
    The converse does not hold in general: in some frames, even completely additive formulas
    may fail to be equivalent to any positive formula.
    \begin{example}
        \label{ex:c2}
        Let $\Lambda$ be the logic of the frame $\mathcal{F} = (\{0, 1\}, \{0, 1\}^2)$, 
        $\varphi$ be the formula $\Box p \vee \neg p \wedge \Dmd p$. 
        It is easy to see that $\varphi_\mathcal{F}$ is (completely) additive, whence $\varphi$ is $\Lambda$-additive. 
        At the same time, it was shown in~\cite[Proposition 4.8]{Dvo26} that $\varphi$ is not equivalent 
        to any $p$-positive formula in $\Lambda$.
    \end{example}
    Notice that $\Lambda$ from Example~\ref{ex:c2} has many good properties: it is tabular, 
    its class of frames is first-order definable, it even has the Craig interpolation property (CIP)~\cite{Maks80}
    but lacks the Lyndon interpolation property (LIP)~\cite{Maks82}. 
    The last fact turns out to be essential: for logics with LIP, we have the following
    analog of the classical Lyndon's result~\cite{Lynd59Hom}:
    \begin{theorem}[{\cite[Theorem 4.4]{Dvo26}}]
      \label{t:lind}
      Let $\Lambda$ be a normal logic with LIP. Then a formula
      $\varphi$ is monotone in $\Lambda$ iff it is $\Lambda$-equivalent to some $p$-positive
      formula.
    \end{theorem}
    It is well known that the logics we consider in this paper ($\mathrm{K}$, $\mathrm{GL}$,
    $\mathrm{S4}$, $\mathrm{S5}$, and $\mathrm{Grz}$) have LIP~\cite{GabMaks05}. Therefore, we will
    look for additive formulas only among $p$-positive ones. 

\section{General results on additive formulas}
    \label{s:gen}
  We say that $A \subseteq \mathrm{Fm}(p,\vec r)$ is a \emph{complete set of (normal) $\Lambda$-additive
  formulas with parameters $\vec r$} if every $\alpha \in A$ is (normal) $\Lambda$-additive 
  and, for each (normal)
  $\Lambda$-additive formula $\varphi\in{\rm Fm}(p,\vec r)$, there is $\alpha \in A$ such that
  $\varphi \sim_\Lambda \alpha$. The goal of our work is to provide syntactic descriptions of some
  complete sets of $\Lambda$-additive and normal $\Lambda$-additive formulas. 
  These two tasks are in fact equivalent: 
  \begin{lemma}
    If $A$ is a complete set of $\Lambda$-additive formulas, then
    $\{\alpha \in A \mid \alpha_\bot \sim_\Lambda \bot\}$ is a complete set of normal
    $\Lambda$-additive formulas. Conversely, if $N$ is a complete set of normal $\Lambda$-additive
    formulas, then $\{\alpha \vee \nu \mid \alpha \in N, \nu \in {\rm Fm}(\vec r)\}$ is a complete set of
    $\Lambda$-additive formulas.
  \end{lemma}
  \begin{proof}
    The first claim is trivial. It is also easy to see that $\alpha \vee \nu$ is additive in
    $\Lambda$, whenever $\alpha$ is additive. Suppose that $\varphi$ is additive in
    $\Lambda$. Since $\varphi$ is monotone in $\Lambda$,
    $\varphi \sim_\Lambda (\varphi \wedge \neg\varphi_\bot) \vee \varphi_\bot$. It remains to notice
    that $\varphi \wedge \neg\varphi_\bot$ is normal $\Lambda$-additive and $\varphi_\bot\in{\rm Fm}(\vec r)$.
  \end{proof}
  The following lemma can be easily checked:
  \begin{lemma}
    \label{l:add-closed}
    Let $\Lambda$ be a normal logic. Then $\bot$ and $p$ are normal $\Lambda$-additive formulas. If
    $\alpha, \beta$ are normal $\Lambda$-additive formulas, $\varkappa \in {\rm Fm}(\vec r)$, then
    $\alpha \vee \beta$, $\varkappa \wedge \alpha$, and $\Dmd\alpha$ are also normal
    $\Lambda$-additive.
  \end{lemma}
  For $\varkappa_0, \dots, \varkappa_n \in {\rm Fm}(\vec r)$, consider the formula
  \begin{gather*}
    \delta_{\vec \varkappa} \defeq \varkappa_0 \wedge \Dmd(\varkappa_1 \wedge \dots \wedge
    \Dmd(\varkappa_n \wedge p) \dots).
  \end{gather*}
  Such formulas are called \emph{$\delta$-formulas}. 
  Notice that $\delta_{\vec\varkappa}$ is true at $w$ in a Kripke model $(W,R,\vartheta)$ 
  iff there is a sequence of worlds $v_0,\dots,v_n$ such that 
  $w=v_0\mathrel{R}v_1\dots\mathrel{R}v_n$, $v_i\in\vartheta(\varkappa_i)$ for $i\leq n$,
  and $v_n\in\vartheta(p)$.
  Denote by $\Delta$ the set of all
  $\delta$-formulas. We write $\Delta_0 \subset_{\rm fin} \Delta$ if $\Delta_0$ is a finite subset
  of $\Delta$. From Lemma~\ref{l:add-closed}, we immediately obtain
  \begin{corollary}
    Let $\Lambda$ be a normal logic, $\Delta_0 \subset_{\rm fin} \Delta$, $\nu \in {\rm Fm}(\vec r)$. Then
    $\bigvee\Delta_0$ is normal $\Lambda$-additive and $\bigvee\Delta_0 \vee \nu$ is
    $\Lambda$-additive.
  \end{corollary}
  One of the difficulties in characterizing additive formulas is that 
  complete sets of additive formulas are not preserved under extensions or weakenings of logics. 
  Moreover, even if the same set $A$ constitutes a complete set of additive formulas for two logics 
  $\Lambda_1$ and $\Lambda_2$ with $\Lambda_1 \subset \Lambda_2$, $A$ need not be a complete set of additive formulas 
  for any logic between $\Lambda_1$ and $\Lambda_2$:
  \begin{example}
    We will show that, in both $\rm K$ and $\rm GL$, every additive formula is equivalent to a formula from the set 
    $A_0 \defeq \{\bigvee\Delta_0 \vee \nu \mid \Delta_0 \subset_{\rm fin} \Delta, \nu \in {\rm Fm}(\vec r) \}$.
    At the same time, ${\rm K}\subset{\rm K4}\subset{\rm GL}$ 
    and the formula $\beta$ from Corollary~\ref{c:k4-dbd} is $\rm K4$-additive but is not 
    $\rm K4$-equivalent to a formula from $A_0$, since formulas from $A_0$ are completely additive in all frames.
    $\beta$ does not appear in the description of $\rm K$- and $\rm GL$-additive formulas, 
    since it is not $\rm K$-additive and it is $\rm GL$-equivalent to a $\delta$-formula $\Dmd(\Box\bot\land p)$.
  \end{example}
  \begin{definition}
    Let $\Lambda$ be a normal modal logic, $\mathcal{D}$ be a class of general frames. We say that
    $\mathcal{D}$ is \emph{good for $\Lambda$} if $\mathrm{Log}\, \mathcal{D} = \Lambda$ and, for
    all $(\mathcal{F}, P) \in \mathcal{D}$, $\mathcal{F}$ is a $\Lambda$-frame
    (i.e., $\mathcal{F} \vDash \Lambda$).
  \end{definition}
  The following lemma trivially follows from the basic properties of p-morphisms:
  \begin{lemma}
    \label{l:gen-D}
    Suppose that $\Lambda = \mathrm{Log}\, \mathcal{C}$ for a class of Kripke frames $\mathcal{C}$
    and, for each $\mathcal{F} \in \mathcal{C}$, we have defined a $\Lambda$-frame
    $\mathcal{F}^\circ$ and a p-morphism $\pi : \mathcal{F}^\circ \twoheadrightarrow \mathcal{F}$.
    Then the class
    \begin{gather*}
      \pi^{-1}\mathcal{C} \defeq \bigl\{(\mathcal{F}^\circ, \pi^{-1}\mathcal{P}(W)) \bigm| \mathcal{F} =
      (W, R) \in \mathcal{C} \bigr\}
    \end{gather*}
    is good for $\Lambda$.
  \end{lemma}
  For a class of finite general frames $\mathcal{D}$, denote by $\mathcal{D}^\dagger$ the class of
  Kripke models
  \begin{gather*}
    \bigl\{(\mathcal{F}, \vartheta) \bigm| (\mathcal{F}, P) \in \mathcal{D},\; |\vartheta(p)| \leq 1,\; 
    \text{and } \vartheta(r) \in P \text{ for every parameter }r \bigr\}.
  \end{gather*}
  Notice that $\mathcal{D}^\dagger \vDash \varphi(p, \vec r) \leftrightarrow \psi(p, \vec r)$ iff,
  for all $\mathcal{G} = (\mathcal{F}, P)\in\mathcal D$, $\vec V \in P^n$,
  \[
    \varphi_\mathcal{F}(\emptyset, \vec V) = \psi_\mathcal{F}(\emptyset, \vec V)
    \quad\text{and}\quad
    \varphi_\mathcal{F}(\{u\}, \vec V) = \psi_\mathcal{F}(\{u\}, \vec V)
    \;\;\text{for every world } u \text{ in }\mathcal F
  \]
  \begin{lemma}
    \label{l:gen-2} Let $\Lambda$ be a normal logic, $\mathcal{D}$ be a good class of
    finite general frames for $\Lambda$, $\alpha$ and $\beta$ be $\Lambda$-additive formulas. Then
    $\mathcal{D}^\dagger \vDash \alpha \leftrightarrow \beta \Leftrightarrow \Lambda \vdash \alpha
    \leftrightarrow \beta$.
  \end{lemma}
  \begin{proof}
    Suppose that $\Lambda \nvdash \alpha\leftrightarrow \beta$.
    Since ${\rm Log\,}\mathcal{C}=\Lambda$, there is
    $(\mathcal{F}, P) \in \mathcal{D}$ and valuation $\vartheta : \mathcal{V} \to P$ such that
    $\vartheta(\alpha) \neq \vartheta(\beta)$. Let $U \defeq \vartheta(p)$ and
    $\vec V \defeq \vartheta(\vec r)$. If $U = \emptyset$, then
    $(\mathcal{F}, \vartheta) \in \mathcal{D}^\dagger$, whence
    $\mathcal{D}^\dagger \nvDash \alpha \leftrightarrow \beta$. Otherwise, since $\alpha$ and
    $\beta$ are $\Lambda$-additive and $\mathcal{F}$ is finite,
    \begin{gather*}
      \vartheta(\alpha) = \alpha_\mathcal{F}(U, \vec V) = \bigcup_{u \in U}\alpha_\mathcal{F}(\{u\},
      \vec V)\quad\text{and}\quad
      \vartheta(\beta) = \beta_\mathcal{F}(U, \vec V) = \bigcup_{u \in U}\beta_\mathcal{F}(\{u\},
      \vec V).
    \end{gather*}
    Therefore, $\alpha_\mathcal{F}(\{u\}, \vec V) \neq \beta_\mathcal{F}(\{u\}, \vec V)$ for some $u\in U$ and
    $\mathcal{D}^\dagger \nvDash \alpha \leftrightarrow \beta$.

    The converse implication is trivial, since for every $(\mathcal{F},P)\in\mathcal{D}$, $\mathcal{F}$ is a $\Lambda$-frame.
  \end{proof}
  Suppose that $\Lambda$ has LIP. Then, by Theorem~\ref{t:lind}, all monotone formulas, and even more so all
  additive ones, are $\Lambda$-equivalent to $p$-positive formulas. 
  Let $\varphi\sim_{\mathcal{D}^\dagger}\psi:\Leftrightarrow\mathcal{D}^\dagger\vDash\varphi\leftrightarrow\psi$.
  It is easy to check that $\sim_{\mathcal{D}^\dagger}$
  is a congruence on the algebra of $p$-positive formulas
  $\mathfrak{Fm}^+(p,\vec r) \defeq (\mathrm{Fm}^+(p,\vec r), \bot, \top, \wedge, \vee, \Dmd, \Box)$. The quotient
  structure $\mathfrak{Fm}^+(p,\vec r)/{\sim}_{\mathcal{D}^\dagger}$ is a modal lattice (i.e., a distributive
  lattice with dual operators $\Dmd$ and $\Box$). By Lemma~\ref{l:gen-2}, each equivalence class in
  $\mathrm{Fm}^+(p,\vec r)/{\sim}_{\mathcal{D}^\dagger}$ contains at most one $\Lambda$-additive formula up
  to $\Lambda$-equivalence. In particular, we have the following:
  \begin{corollary}
    \label{l:gen-3} Let $\Lambda$ be a normal modal logic with LIP, $\mathcal{D}$ be a good
    for $\Lambda$ class of finite general frames, $A\subset{\rm Fm}(p,\vec r)$ be a set of $\Lambda$-additive formulas. Suppose that,
    for each $p$-positive formula $\varphi(p,\vec r)$, there is $\alpha \in A$ such that
    $\mathcal{D}^\dagger \vDash \varphi \leftrightarrow \alpha$. 
    Then $A$ is a complete set of $\Lambda$-additive formulas with parameters $\vec r$.
  \end{corollary}

\section{Additive formulas in S5}
    \label{s:S5}
  Consider the following class of frames:
  \begin{gather*}
    \mathcal{C}_\mathrm{S5} \defeq \{ (W, W \times W) \mid W \subseteq\omega\text{ is a finite non-empty set} \}.
  \end{gather*}
  It is well known that $\mathrm{Log}\, \mathcal{C}_\mathrm{S5} = \mathrm{S5}$. Let
  $B \defeq \{0, 1\}$. For a frame $\mathcal{F} = (W, W^2) \in \mathcal{C}_\mathrm{S5}$, we put
  $\mathcal{F}^\circ \defeq (W \times B, (W \times B)^2)$. Clearly
  $\mathcal{F}^\circ \vDash \mathrm{S5}$ and the mapping $\pi : (a, j) \mapsto a$ is a
  p-morphism from $\mathcal{F}^\circ$ onto $\mathcal{F}$. By Lemma~\ref{l:gen-D},
  the class of general frames $\mathcal{D}_\mathrm{S5} \defeq \pi^{-1}\mathcal{C}_\mathrm{S5}$ 
  is good for $\mathrm{S5}$.

  Consider the following sets of formulas
  \begin{gather*}
    \Gamma_\mathrm{S5} \defeq {\rm Fm}(\vec r) \cup \{\lambda \wedge p \mid \lambda \in {\rm Fm}(\vec r)\} \cup
    \{\mu \wedge \Dmd(\nu \wedge p) \mid \mu, \nu \in {\rm Fm}(\vec r)\}.\\
    A_\mathrm{S5} \defeq \biggl\{  (\lambda \wedge p) \vee \bigvee_{i<k}\bigl(\mu_i
    \wedge \Dmd(\nu_i \wedge p)\bigr) \vee \varkappa
    \biggm| k \in \omega;\;\varkappa, \lambda, \mu_i, \nu_i \in {\rm Fm}(\vec r) \biggr\}.
  \end{gather*}
  Notice that $A_\mathrm{S5}\subset A_0$ and
  $A_\mathrm{S5}$ contains all disjunctions of formulas from $\Gamma_\mathrm{S5}$ up to $\rm K$-equivalence.
  We will show that $A_\mathrm{S5}$ is the complete set of $\mathrm{S5}$-additive formulas.
  \begin{lemma}
    \label{S5-c} For $\alpha, \beta \in \Gamma_\mathrm{S5}$, there is
    $\gamma \in \Gamma_\mathrm{S5}$ such that
    $\mathcal{D}^\dagger_\mathrm{S5} \vDash \alpha \wedge \beta \leftrightarrow \gamma$.
  \end{lemma}
  \begin{proof}
    If $\alpha = \lambda_1 \wedge p$ and $\beta = \mu_2 \wedge \Dmd(\nu_2 \wedge p)$, 
    where $\lambda_1,\mu_2,\nu_2\in{\rm Fm}(\vec r)$, then
    \begin{gather*}
      \mathcal{D}^\dagger_\mathrm{S5} \vDash \alpha \wedge \beta \leftrightarrow (\lambda_1 \wedge
      \mu_2 \wedge \nu_2) \wedge p.
    \end{gather*}
    Indeed, if $\alpha \wedge \beta$ is true at $w$ in
    $\mathcal{M} = (\mathcal{F}^\circ, \vartheta) \in \mathcal{D}^\dagger_\mathrm{S5}$, then
    $w \in \vartheta(p)$ and there is a world $v \in \vartheta(\nu_2 \wedge p)$. 
    Since $|\vartheta(p)| \leq 1$, $w = v$ and the right-hand side of the equivalence is true at $w$. 
    The converse implication is trivial.

    Similarly, if $\alpha = \mu_1 \wedge \Dmd(\nu_1 \wedge p)$ and
    $\beta = \mu_2 \wedge \Dmd(\nu_2 \wedge p)$, where $\mu_1,\mu_2,\nu_1,\nu_2\in{\rm Fm}(\vec r)$, then
    \begin{gather*}
      \mathcal{D}^\dagger_\mathrm{S5} \vDash \alpha \wedge \beta \leftrightarrow (\mu_1 \wedge
      \mu_2) \wedge \Dmd(\nu_1 \wedge \nu_2 \wedge p).
    \end{gather*}
    All other cases are either trivial or symmetric to those already considered.
  \end{proof}
  \begin{lemma}
    \label{S5-d} For $\alpha \in \Gamma_\mathrm{S5}$, there is $\beta \in \Gamma_\mathrm{S5}$ such
    that $\mathrm{S5} \vdash \Dmd\alpha \leftrightarrow \beta$.
  \end{lemma}
  \begin{proof}
    Notice that $\mathrm{S5} \vdash \Dmd(\mu \wedge \Dmd(\nu \wedge p)) \leftrightarrow \Dmd\mu
    \wedge \Dmd(\nu \wedge p)$. All other cases are trivial.
  \end{proof}
  \begin{lemma}
    \label{S5-b} For $\alpha \in A_\mathrm{S5}$, there is $\beta \in A_\mathrm{S5}$ such that
    $\mathcal{D}^\dagger_\mathrm{S5} \vDash \Box\alpha \leftrightarrow \beta$.
  \end{lemma}
  \begin{proof}
    Suppose that $\alpha = (\lambda \wedge p) \vee \bigvee_{i<k}\mu_i \wedge
    \Dmd(\nu_i \wedge p) \vee \varkappa$ for some ${k \in \omega}$ and
    $\varkappa, \lambda, \mu_i, \nu_i \in {\rm Fm}(\vec r)$. 
    Let $\un k:=\{0,\dots,k-1\}$.
    For $I \subseteq \un k$, we put
    $\mu^I \defeq \bigvee_{i \in I}\mu_i$ and $\nu_I \defeq \bigwedge_{i \in I}\nu_i$. 
    Consider the formula
    \begin{gather*}
      \beta \defeq \bigvee_{\emptyset \neq I \subseteq \un k} \bigl(
      \Box(\varkappa \vee \mu^I) \wedge \Dmd(\nu_I \wedge p) \bigr)\vee \Box\varkappa.
    \end{gather*}
    Clearly, $\mathrm{S5} \vdash \beta \to \Box\alpha$.
    For the converse, suppose that $\Box\alpha$ is true at $w$ in
    $\mathcal{M} = (\mathcal{F}^\circ, \vartheta) \in \mathcal{D}^\dagger_\mathrm{S5}$. If
    $\vartheta(p) = \emptyset$, then $\mathcal{M}, w \vDash \Box\varkappa$. Otherwise, there is a
    world $v$ such that $\vartheta(p) = \{v\}$. Let
    $I \defeq \{i < k \mid \mathcal{M}, v \vDash \nu_i\}$. Clearly,
    $\mathcal{M}, w \vDash \Dmd(\nu_I \wedge p)$ and
    $\mathcal{M} \vDash \neg\Dmd(\nu_i \wedge p)$ for $i \in \un k \setminus I$.

    Suppose that $\mathcal{M}, w \nvDash \Box(\varkappa \vee \mu^I)$. Then there is a world $u$ such
    that $\mathcal{M}, u \nvDash \varkappa \vee \mu^I$. Let $u' \defeq (a, 1-j)$, where $u = (a, j)$.
    Since $\pi u = \pi u' = a$,
    $\mathcal{M}, u' \nvDash \varkappa \vee \mu^I$. At least one of the worlds $u$ and $u'$ is
    distinct from $v$. Without loss of generality, we can assume that it is $u$. Then
    $\mathcal{M}, u \nvDash \lambda \wedge p$, whence $\mathcal{M}, u \nvDash \alpha$ and
    $\mathcal{M}, w \nvDash \Box\alpha$. Contradiction.
  \end{proof}

  \begin{theorem}
    A formula $\varphi$ is $\mathrm{S5}$-additive iff it is $\mathrm{S5}$-equivalent to a formula
    from $A_\mathrm{S5}$.
  \end{theorem}
  \begin{proof}
    By Lemma~\ref{l:gen-3}, it is sufficient to prove that every $\varphi\in{\rm Fm}^+(p,\vec r)$ is
    equivalent to some formula from $A_\mathrm{S5}$ in $\mathcal{D}^\dagger_\mathrm{S5}$.
    We proceed by induction on the construction of $\varphi$. The base case and the induction step for
    $\vee$ are trivial. The induction step for $\wedge$ follows, by distributivity, from Lemma~\ref{S5-c}.
    The induction step for $\Dmd$ follows, by additivity, from Lemma~\ref{S5-d}. The induction step for $\Box$
    follows from Lemma~\ref{S5-b}.
  \end{proof}
  \begin{corollary}
    $\varphi$ is a normal additive formula in $\mathrm{S5}$ iff it is equivalent to a formula from
    $A_\mathrm{S5}$ with $\varkappa = \bot$.
  \end{corollary}
  \begin{corollary}
    \label{c:S5-add}
    There are exactly four $\mathrm{S5}$-additive formulas without parameters up to
    $\mathrm{S5}$-equivalence: $\bot$, $p$, $\Dmd p$, and $\top$. Three of them are normal: $\bot$,
    $p$, and $\Dmd p$.
  \end{corollary}

\section{Additive formulas in K and GL}
    \label{s:K-GL}
  For a set $\Sigma$, denote by $\Sigma^*$ the set of all finite sequences
  $\langle s_i\rangle_{i<n}$ with $s_i \in \Sigma$, including the empty one~$\langle\rangle$, 
  $\Sigma^+ \defeq \Sigma^* \setminus \{ \langle\rangle \}$. For
  two sequences $\vec a, \vec b \in \Sigma^*$, denote by $\vec a\vec b$ their concatenation. We consider the
  following relation on $\Sigma^*$:
  \begin{gather*}
    \vec a \lessdot \vec b :\Leftrightarrow \E{ c \in \Sigma}(\vec a = \vec b \langle c\rangle).
  \end{gather*}
  Denote by $<$ the transitive closure of $\lessdot$. Notice that $<$ is a strict partial order.
  \begin{definition}
    A non-empty set $S \subseteq \Sigma^*$ is a \emph{tree skeleton} if
    \begin{gather*}
      \A{\vec a, \vec b \in \Sigma^*}(\vec a < \vec b \wedge \vec b \in S \Rightarrow \vec a \in S).
    \end{gather*}
  \end{definition}
  Consider the following mappings:
  \begin{gather*}
    \begin{aligned}
      \pi : (\Sigma \times B)^* &\to \Sigma^*\\
      \langle(a_i, j_i)\rangle_{i < n} &\mapsto \langle a_i\rangle_{i < n},
    \end{aligned}
    \qquad\text{and}\qquad
    \begin{aligned}
      \iota : (\Sigma \times B)^+ &\to (\Sigma \times B)^+\\
      \langle(a_i,j_i)\rangle_{i\leq n} &\mapsto
      \langle(a_i,j_i)\rangle_{i<n}\langle(a_n,1-j_n)\rangle.
    \end{aligned}
  \end{gather*}
  For a tree skeleton $S \subseteq \Sigma^*$, let $S^\circ \defeq \pi^{-1}S$. Notice that
  $S^\circ \subseteq (\Sigma \times B)^*$ is also a tree skeleton and, for
  $v \in (\Sigma \times B)^+$, $\pi(\iota v) = \pi(v)$, whence $\iota v \in S^\circ$ iff
  $v \in S^\circ$.

  Consider the following classes of frames:
  \begin{align*}
    \mathcal{C}_\mathrm{K} &\defeq \{ (S, \lessdot) \mid S\subseteq\omega^* \text{ is a finite tree skeleton} \},\\
    \mathcal{C}_\mathrm{GL} &\defeq \{ (S, <) \mid S\subseteq\omega^* \text{ is a finite tree skeleton} \}.
  \end{align*}
  Strictly speaking, we should write $(S,{\lessdot}|_S)$ and $(S,{<}|_S)$, 
  but we drop the restriction for notational convenience.
  It is well known that $\mathrm{Log}\, \mathcal{C}_\mathrm{K} = \mathrm{K}$ and
  $\mathrm{Log}\, \mathcal{C}_\mathrm{GL} = \mathrm{GL}$. For
  $\mathcal{F} = (S, \lessdot) \in \mathcal{C}_\mathrm{K}$, we put
  $\mathcal{F}^\circ \defeq (S^\circ, \lessdot)$. For
  $\mathcal{F} = (S, <) \in \mathcal{C}_\mathrm{GL}$, we put
  $\mathcal{F}^\circ \defeq (S^\circ, <)$. Clearly, for
  $\Lambda \in \{\mathrm{K}, \mathrm{GL}\}$ and $\mathcal{F} \in \mathcal{C}_\Lambda$,
  $\mathcal{F}^\circ$ is a $\Lambda$-frame and
  $\pi : \mathcal{F}^\circ \twoheadrightarrow \mathcal{F}$ is a p-morphism.
  By Lemma~\ref{l:gen-D}, the class of
  general frames $\mathcal{D}_\Lambda \defeq \pi^{-1}\mathcal{C}_\Lambda$ is good for
  $\Lambda$.

  We are going to show that $A_0 = \{ \bigvee\Delta_0 \vee \varkappa \mid \Delta_0
    \subset_{\rm fin} \Delta, \varkappa \in {\rm Fm}(\vec r) \}$
  is a complete set of additive  formulas both in $\mathrm{K}$ and in $\mathrm{GL}$.
  \begin{lemma}
    \label{K-c} For $\vec \mu = \langle\mu_i\rangle_{i \leq m} \in ({\rm Fm}(\vec r))^{m+1}$ and
    $\vec \nu = \langle\nu_i\rangle_{i \leq n} \in ({\rm Fm}(\vec r))^{n+1}$,
    \begin{itemize}
      \item if $m \neq n$, then $\mathcal{D}^\dagger_\mathrm{K} \vDash \delta_{\vec \mu} \wedge
            \delta_{\vec \nu} \leftrightarrow \bot$;
      \item if $m = n$, then $\mathcal{D}^\dagger_\mathrm{K} \vDash \delta_{\vec \mu} \wedge
            \delta_{\vec \nu} \leftrightarrow \delta_{\vec \mu \wedge \vec \nu}$, where
            $\vec \mu \wedge \vec \nu = \langle\mu_i \wedge \nu_i\rangle_{i \leq m}$.
    \end{itemize}
  \end{lemma}
  \begin{proof}
    Suppose that $\mathcal{M} = (\mathcal{F}, \vartheta) \in \mathcal{D}^\dagger_\mathrm{K}$, and
    $\delta_{\vec \mu} \wedge \delta_{\vec \nu}$ is true at $w$. Then there are sequences of worlds
    $\langle u_i\rangle_{i \leq m}$ and $\langle v_i\rangle_{i \leq n}$ such that $u_0 = v_0 = w$,
    $u_0 \lessdot u_1 \lessdot \dots \lessdot u_m$, $v_0 \lessdot v_1 \lessdot \dots \lessdot v_n$,
    $u_i \in \vartheta(\mu_i)$, $v_i \in \vartheta(\nu_i)$,
    $u_m \in \vartheta(p)$, and $v_n \in \vartheta(p)$. Since $|\vartheta(p)| \leq 1$, $u_m = v_n$.
    Moreover, since every world in $\mathcal{F}$ has at most one predecessor, 
    $m = n$ and $u_i = v_i$ for all $i \leq n$, whence
    $\mathcal{M}, w \vDash \delta_{\vec \mu \wedge \vec \nu}$. 
    The remaining implications are trivial.
  \end{proof}
  For $\mathcal{D}^\dagger_{\rm GL}$, we need a more involved argument:
  \begin{lemma}
    \label{l:dd} For all $\vec \mu = \langle\mu_i\rangle_{i \leq m} \in ({\rm Fm}(\vec r))^{m+1}$ and
    $\vec \nu = \langle\nu_i\rangle_{i \leq n} \in ({\rm Fm}(\vec r))^{n+1}$,
    \begin{gather*}
        \mathcal{D}^\dagger_\mathrm{GL} \vDash \Dmd\delta_{\vec \mu} \wedge \Dmd\delta_{\vec \nu}
        \leftrightarrow \Dmd(\delta_{\vec \mu} \wedge \delta_{\vec \nu}) \vee
        \Dmd(\delta_{\vec \mu} \wedge \Dmd\delta_{\vec \nu}) \vee \Dmd(\Dmd\delta_{\vec \mu} \wedge
        \delta_{\vec \nu}).
    \end{gather*}
  \end{lemma}
  \begin{proof}
      The right-to-left implication is derivable in $\mathrm{K}$. 
      Suppose that $\Dmd\delta_{\vec \mu} \wedge \Dmd\delta_{\vec \nu}$ is true at $w$
      in a model $\mathcal{M} = (S^\circ, <, \vartheta) \in \mathcal{D}^\dagger_\mathrm{GL}$. 
      Then there are
      $u, v \in S^\circ$ such that $w < u$, $w < v$, $\mathcal{M}, u \vDash \delta_{\vec \mu}$, and
      $\mathcal{M}, v \vDash \delta_{\vec \nu}$. Since $|\vartheta(p)| \leq 1$, there is $x \in S^\circ$
      such that $u \leq x$, $v \leq x$, and $\vartheta(p) = \{x\}$. 
      It is easy to see that in this case $u$ and $v$ are
      comparable, i.e., $u = v$, $u < v$, or $v < u$. 
      Thus, at least one of the disjuncts of the right-hand side of
      the equivalence is true at $w$.
  \end{proof}
  \begin{lemma}
    \label{GL-c} For each $\alpha, \beta \in \Delta$, there is a set
    $\Delta_0 \subset_{\rm fin} \Delta$ such that
    $\mathcal{D}^\dagger_\mathrm{GL} \vDash \alpha \wedge \beta \leftrightarrow \bigvee\Delta_0$.
  \end{lemma}
  \begin{proof}
    Let
    $\vec \mu = \langle\mu_i\rangle_{i \leq m} \in ({\rm Fm}(\vec r))^{m+1}$ and
    $\vec \nu = \langle\nu_i\rangle_{i \leq n} \in ({\rm Fm}(\vec r))^{n+1}$. We need to construct a set
    $\Delta_0 \subset_{\rm fin} \Delta$ such that $\mathcal{D}^\dagger_\mathrm{GL} \vDash \delta_{\vec
    \mu} \wedge \delta_{\vec \nu} \leftrightarrow \bigvee\Delta_0$. We proceed by induction on $m + n$.
    Without loss of generality, we may assume that $m \leq n$. The case $m = n = 0$ is trivial.
    
    If $m = 0$ and $n > 0$, then $\delta_{\vec \mu} = \mu_0 \wedge p$ and
    $\delta_{\vec \nu} = \nu_0 \wedge \Dmd\delta_{\bar \nu}$, where
    $\bar \nu \defeq \langle\nu_1, \dots, \nu_n\rangle$. It is easy to see that
    $\mathrm{K4} \vdash \Dmd\delta_{\bar \nu} \to \Dmd p$, whence
    $\mathrm{K4} \vdash \delta_{\vec \mu} \wedge \delta_{\vec \nu} \to p \wedge \Dmd p$. At the same
    time, since $\mathrm{GL}$-frames are irreflexive and $p$ holds in at most one world in any model from
    $\mathcal{D}^\dagger_\mathrm{GL}$, we have $p \wedge \Dmd p \sim_{\mathcal{D}^\dagger_\mathrm{GL}} \bot$.
    Thus, $\delta_{\vec \mu} \wedge \delta_{\vec \nu} \sim_{\mathcal{D}^\dagger_\mathrm{GL}}
    \bigvee\emptyset$.
    
    If $n, m > 0$, then $\delta_{\vec \mu} = \mu_0 \wedge \Dmd\delta_{\bar \mu}$ and
    $\delta_{\vec \nu} = \nu_0 \wedge \Dmd\delta_{\bar \nu}$, where
    $\bar \mu \defeq \langle\mu_1, \dots, \mu_m\rangle$ and
    $\bar \nu \defeq \langle\nu_1, \dots, \nu_n\rangle$. By Lemma~\ref{l:dd},
    \begin{gather}
      \mathcal{D}^\dagger_\mathrm{GL} \vDash \Dmd\delta_{\bar \mu} \wedge \Dmd\delta_{\bar \nu}
      \leftrightarrow \Dmd(\delta_{\bar \mu} \wedge \delta_{\bar \nu}) \vee
      \Dmd(\delta_{\bar \mu} \wedge \Dmd\delta_{\bar \nu}) \vee \Dmd(\Dmd\delta_{\bar \mu} \wedge
      \delta_{\bar \nu}). \label{eq:1}
    \end{gather}
    Notice that $\Dmd\delta_{\bar \mu} \sim_\mathrm{K} \delta_{\langle\top,\mu_1,\dots,\mu_m\rangle}$.
    Therefore, by the induction hypothesis, there are
    $\Delta_1, \Delta_2, \Delta_3 \subset_{\rm fin} \Delta$ such that
    \begin{gather}
      \delta_{\bar \mu} \wedge \delta_{\bar \nu} \sim_{\mathcal{D}^\dagger_\mathrm{GL}}
      \bigvee\Delta_1,\quad
      \delta_{\bar \mu} \wedge \Dmd\delta_{\bar \nu} \sim_{\mathcal{D}^\dagger_\mathrm{GL}}
      \bigvee\Delta_2,\quad\text{and}\quad \Dmd\delta_{\bar \mu} \wedge \delta_{\bar \nu}
      \sim_{\mathcal{D}^\dagger_\mathrm{GL}} \bigvee\Delta_3. \label{eq:2}
    \end{gather}
    Combining (\ref{eq:1}) and (\ref{eq:2}), we obtain
    \begin{gather*}
      \mathcal{D}^\dagger_\mathrm{GL} \vDash \Dmd\delta_{\bar \mu} \wedge \Dmd\delta_{\bar \nu}
      \leftrightarrow \Dmd\bigvee\Delta_1 \vee \Dmd\bigvee\Delta_2 \vee \Dmd\bigvee\Delta_3.
    \end{gather*}
    Thus, for $\Delta_0 \defeq \{\delta_{\langle\mu_0 \wedge \nu_0, \varkappa_0, \dots,
    \varkappa_k\rangle} \mid \delta_{\vec \varkappa} \in \Delta_1 \cup \Delta_2 \cup \Delta_3\}$, we
    have
    \begin{align*}
      \delta_{\vec \mu} \wedge \delta_{\vec \nu}
      \;&=\; (\mu_0 \wedge \Dmd\delta_{\bar \mu}) \wedge (\nu_0 \wedge \Dmd\delta_{\bar \nu})\\
      &\sim_{\mathcal{D}^\dagger_\mathrm{GL}} \mu_0 \wedge \nu_0 \wedge \Bigl(\Dmd\bigvee\Delta_1 \vee
      \Dmd\bigvee\Delta_2 \vee \Dmd\bigvee\Delta_3\Bigr)\\
      &\sim_\mathrm{K}\; \bigvee\Delta_0.
    \end{align*}
  \end{proof}
  \begin{lemma}
    \label{K-GL-b} Let $\varphi$ be a $p$-positive formula. 
    Then $\Box\varphi \leftrightarrow \Box\varphi_\bot$
    is true in both $\mathcal{D}^\dagger_\mathrm{K}$
    and $\mathcal{D}^\dagger_\mathrm{GL}$.
  \end{lemma}
  \begin{proof}
    Since $\varphi$ is $p$-positive,
    $\mathrm{K} \vdash \Box\varphi_\bot \to \Box\varphi$. Suppose that $\mathcal{M} = (S^\circ, R,
    \vartheta) \in \mathcal{D}^\dagger_\mathrm{K} \cup \mathcal{D}^\dagger_\mathrm{GL}$, and
    $\Box\varphi_\bot$ is false at $w$. Then there is a world $v \in R(w)$ such that
    $\mathcal{M}, v \nvDash \varphi_\bot$. It is easy to see from the definition of
    $\mathcal{D}_\mathrm{K}$ and $\mathcal{D}_\mathrm{GL}$ that $\iota v \in R(w)$,
    $\varphi_\bot$ is false at $\iota v$, and the submodels generated by $v$ and $\iota v$
    are disjoint. Since $|\vartheta(p)| \leq 1$, $\varphi \leftrightarrow \varphi_\bot$ holds
    either at $v$ or at $\iota v$. Thus, $\varphi$ is false at that world and
    $\mathcal{M}, w \nvDash \Box\varphi$.
  \end{proof}
  \begin{theorem}
    \label{t:K-GL}
    For $\Lambda \in \{\mathrm{K}, \mathrm{GL}\}$, a formula $\varphi$ is $\Lambda$-additive iff it
    is $\Lambda$-equivalent to a formula from $A_0$.
  \end{theorem}
  \begin{proof}
    By Lemma~\ref{l:gen-3}, it is sufficient to prove that every p-positive formula $\varphi$ is equivalent
    to a formula from $A_0$ in $\mathcal{D}^\dagger_\Lambda$. We proceed by induction on
    the construction of $\varphi$. The base case and the induction steps for $\vee$ and $\Dmd$ are trivial.
    The induction step for $\wedge$, by distributivity, follows from Lemmas~\ref{K-c} and~\ref{GL-c}. 
    The induction step for $\Box$ follows from Lemma~\ref{K-GL-b}.
  \end{proof}
  \begin{corollary}
    A formula $\varphi$ is normal additive in $\mathrm{K}$ or $\mathrm{GL}$ iff it is equivalent
    to a disjunction of $\delta$-formulas.
  \end{corollary}
  \begin{remark}
      \label{r:constr}
      Our proof of Theorem~\ref{t:K-GL}, in contrast to the model-theoretic arguments of~\cite{Benth98,Benth26}, is constructive.
      In fact, we have the following: there exists an algorithm that, given a ${\rm K}$-additive formula $\varphi$,
      computes a formula $\alpha \in A_0$ such that $\varphi\sim_{\rm K}\alpha$.
      This algorithm proceeds in two steps:
      \begin{enumerate}
          \item Given a $\rm K$-additive (and hence monotone in $K$) formula $\varphi$,
          compute a $p$-positive formula $\eta$ such that $\varphi\sim_{\rm K}\eta$.
          \item Given a positive $\rm K$-additive formula $\eta$, compute a formula $\alpha\in A_0$ 
          such that $\eta\sim_{\rm K}\alpha$.
      \end{enumerate}
      The algorithm for the first step follows from the proof of~\cite[Theorem~4.4]{Dvo26} 
      together with the fact that Lyndon interpolants for given formulas can be effectively constructed in ${\rm K}$ 
      (see, e.g.,~\cite[Section 4]{BCI25}).
      The algorithm for the second step can be extracted from the proof of Theorem~\ref{t:K-GL}.
  \end{remark}

\section{Additive formulas in S4 and Grz}
    \label{s:S4-Grz}
  Turning to the description of additive formulas in $\mathrm{S4}$ and $\mathrm{Grz}$, we face some
  difficulties. In these logics, there are additive formulas which 
  are not equivalent to any formula from $A_0$: by Proposition~\ref{p:s4-dbd}, 
  the formula $\alpha=\Diamond\Box\Diamond p$ is $\rm S4$-additive but not completely additive 
  in the class of all $\rm S4$-frames and formulas from $A_0$ are clearly completely additive.
  Moreover, it is easy to see that $A_0$ contains only four parameter-free formulas up to $\rm S4$-equivalence:
  $\bot$, $p$, $\Dmd p$, and $\top$. 
  Clearly, neither $\alpha$ nor $\alpha\lor p$ is equivalent to any of those 
  four formulas even in $\rm Grz$ (though $\alpha\sim_{\rm Grz}\Diamond\Box p$ and $\alpha$ is completely additive in all ${\rm Grz}$-frames).
  
  In this section, we characterize parameter-free additive formulas in $\mathrm{S4}$ and
  $\mathrm{Grz}$. As before, it is sufficient to consider only positive formulas. In fact, there are
  only 37 formulas in ${\rm Fm}^+(p)$ up to $\mathrm{S4}$-equivalence~\cite[Theorem~5.1]{Mor19}, so in principle, the
  problem could be solved by exhaustive search. However, we present a more conceptual proof, which
  we hope can be extended to formulas with parameters. The overall strategy is similar to that used
  in the previous sections, although some of the steps differ:
  \begin{enumerate}
    \item Define a class of frames $\mathcal{D}$ that is good for $\Lambda$.
    \item Describe the modal lattice $\mathfrak{Fm}^+(p)/{\sim}_{\mathcal{D}^\dagger}$.
    \item For each equivalence class in ${\rm Fm}^+(p)/{\sim}_{\mathcal{D}^\dagger}$,
          either exhibit a $\Lambda$-additive representative or prove that there is no 
          $\Lambda$-additive formula in it.
  \end{enumerate}
  Then, by Lemma~\ref{l:gen-2}, the $\Lambda$-additive formulas from the third step form a complete set of 
  parameter-free $\Lambda$-additive formulas.
  To simplify the second step of our plan, we will use the following observation:
  \begin{lemma}
    \label{S4-lat} 
    Suppose that $\Lambda\supseteq\mathrm{S4}$, $\mathcal{D}$ is good for $\Lambda$.
    Then, for all elements $x$ and $y$ of $\mathfrak{Fm}^+(p)/{\sim}_{\mathcal{D}^\dagger}$,
    \begin{itemize}
      \item if $x \leq y \leq \Dmd x$, then $\Dmd y = \Dmd x$;
      \item if $\Box x \leq y \leq x$, then $\Box y = \Box x$.
    \end{itemize}
  \end{lemma}
  \begin{proof}
    Since $\Dmd p\to\Dmd(p\lor q)$ and $\Dmd\Dmd p\leftrightarrow\Dmd p$
    are derivable in $\mathrm{S4}$, they are true in $\mathcal{D}^\dagger$,
    whence the operator $\Dmd$ in $\mathfrak{Fm}^+(p)/{\sim}_{\mathcal{D}^\dagger}$
    is monotone and idempotent. 
    Therefore, $\Dmd x \leq \Dmd y \leq \Dmd\Dmd x = \Dmd x$
    whenever $x\leq y\leq\Dmd x$. The second claim is similar.
  \end{proof}
  \subsection{Additive formulas in Grz}
    It is well known that $\mathrm{Grz}$ is the logic of the following class of frames:
    \begin{gather*}
      \mathcal{C}_\mathrm{Grz} \defeq \{ (S, \leq) \mid S\subseteq\omega^* \text{ is a finite tree skeleton}\},
    \end{gather*}
    where $\leq$ is the reflexive closure of $<$. For
    $\mathcal{F} = (S, \leq) \in \mathcal{C}_\mathrm{Grz}$, we put
    $\mathcal{F}^\circ \defeq (S^\circ, \leq)$. Clearly, $\mathcal{F}^\circ \vDash \mathrm{Grz}$, and
    $\pi : \mathcal{F}^\circ \twoheadrightarrow \mathcal{F}$ is a p-morphism.
    By Lemma~\ref{l:gen-D}, the class of general frames $\mathcal{D}_\mathrm{Grz} \defeq \pi^{-1}\mathcal{C}_\mathrm{Grz}$ is good for
    $\mathrm{Grz}$.
    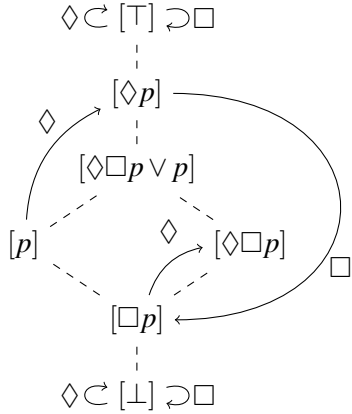
\begin{figure}
        \centering
        \begin{tikzpicture}[
  ord/.style={dashed},
  dx/.style={shift={(#1,0)}},
  dx/.value required,
  dy/.style={shift={(0,#1)}},
  dy/.value required,
  above/.style={dy=1},
  left/.style={dx=-1.5},
  right/.style={dx=1.5},
]


\newcommand{\C}[1]{[#1]}

\node (1) at (0,0) {$\C{\bot}$};
\node (3b) at ([above]1) {$\C{\Box p}$};
\node (4a) at ([left,above]3b) {$\C{p}$};
\node (4b) at ([right,above]3b) {$\C{\Dmd\Box p}$};
\node (5) at ([left,above]4b) {$\C{\Dmd\Box p\lor p}$};
\node (6) at ([above]5) {$\C{\Dmd p}$};
\node (7) at ([above]6) {$\C{\top}$};

\draw[ord] (1) -- (3b) -- (4a) -- (5) -- (6) -- (7);
\draw[ord] (3b) -- (4b) -- (5);

\draw[loop left, looseness=5,->] (1) to node[midway,dx=-.2] {$\Dmd$} (1);
\draw[loop right, looseness=5,->] (1) to node[midway,dx=.2] {$\Box$} (1);
\draw[loop left, looseness=5,->] (7) to node[midway,dx=-.2] {$\Dmd$} (7);
\draw[loop right, looseness=5,->] (7) to node[midway,dx=.2] {$\Box$} (7);

\draw[loop right,looseness=2.5,->,in=0,out=0] (6) to node[dy=-.8]{$\Box$} (3b);
\draw[bend left,looseness=1,->] (4a) to node[label=$\Dmd$]{} (6);
\draw[bend left,looseness=1,->] (3b) to node[label=$\Dmd$]{} (4b);

\end{tikzpicture} 
        \caption{The lattice $\mathfrak{Fm}^+(p)/{\sim}_{\mathcal{D}^\dagger_\mathrm{Grz}}$} 
        \label{fig:Grz}
    \end{figure}
    \begin{proposition}
      The modal lattice $\mathfrak{Fm}^+(p)/{\sim}_{\mathcal{D}^\dagger_\mathrm{Grz}}$ consists of the
      equivalence classes of the formulas from the set
      \begin{gather*}
        \Psi_\mathrm{Grz} \defeq \{ \bot,\; p,\; \top,\; \Dmd p,\; \Box p,\; \Dmd\Box p,\; \Dmd\Box p \vee p \},
      \end{gather*}
      where operations are shown in Figure~\ref{fig:Grz} (missing arrows for $\Dmd$ and $\Box$
      can be restored by Lemma~\ref{S4-lat}).
    \end{proposition}
    \begin{proof}
      One can easily check that all formulas from $\Psi_\mathrm{Grz}$ are pairwise non-equivalent in
      $\mathcal{D}^\dagger_\mathrm{Grz}$. 
      Also, the ordering in Figure~\ref{fig:Grz} corresponds to the ordering in $\mathrm{Grz}$
      (i.e., if $[\alpha]$ is below $[\beta]$ in the figure, then $\mathrm{Grz}\vdash\alpha\to\beta$).
      The join of $[p]$ and $[\Dmd\Box p]$ is clear. 
      We need to check that $[p]\land[\Dmd\Box p] = [\Box p]$ or, more precisely, 
      that $\mathcal{D}^\dagger_\mathrm{Grz}\vDash p\land\Dmd\Box p\to\Box p$ 
      (the converse implication is provable in $\mathrm{Grz}$).
      Suppose that $\mathcal{M} = (S^\circ, \leq, \vartheta) \in \mathcal{D}^\dagger_\mathrm{Grz}$,
      $w\in\vartheta(p\land\Dmd\Box p)$.
      Then there is $v \in S^\circ$ such that $\mathcal{M}, v \vDash \Box p$.
      Since $v\leq v$, $v\in\vartheta(p)$.
      But $|\vartheta(p)|\leq1$, whence $w=v$ and $\mathcal{M}, w \vDash \Box p$.

      The actions of $\Dmd$ and $\Box$ presented in Figure~\ref{fig:Grz} are also mostly clear.
      We only need to check that $\Box[\Dmd p]=[\Box p]$ or, more precisely, 
      that $\mathcal{D}^\dagger_\mathrm{Grz}\vDash\Box\Dmd p\to\Box p$.
      Suppose that $\mathcal{M} = (S^\circ, \leq, \vartheta) \in \mathcal{D}^\dagger_\mathrm{Grz}$
      and $w\in\vartheta(\Box\Dmd p)$. 
      If there is $v \in S^\circ$ such that $w < v$, then $w < \iota v \in S^\circ$ and, 
      like in the proof of Lemma~\ref{K-GL-b}, $\Dmd p$ is false either in $v$ or in $\iota v$. 
      This contradicts $\mathcal{M}, w \vDash \Box\Dmd p$.
      Therefore, $w$ is a maximal element of $S^\circ$, i.e., $w\leq v\Leftrightarrow w = v$. 
      Thus, since $\mathcal{M}, w \vDash \Box\Dmd p$, $w\in\vartheta(p)$ and $\mathcal{M}, w \vDash \Box p$.
    \end{proof}
    The first two steps of our plan are now complete. 
    We also observed above that every formula in $\Psi_{\rm Grz}$ other than $\Box p$ is $\rm Grz$-additive.
    Therefore, it remains to verify that $[\Box p]$ contains no $\rm Grz$-additive representative,
    i.e., no $\rm Grz$-additive formula is $\mathcal{D}^\dagger_\mathrm{Grz}$-equivalent to $\Box p$.
    Notice that this claim is stronger than the simple observation that the formula $\Box p$ itself is not $\rm Grz$-additive.
    \begin{proposition}
      \label{p:Grz-nadd} No $\mathrm{Grz}$-additive formula is
      $\mathcal{D}^\dagger_\mathrm{Grz}$-equivalent to $\Box p$.
    \end{proposition}
    \begin{proof}
      Suppose that $\varphi \in \mathrm{Fm}(p)$ is $\mathcal{D}^\dagger_\mathrm{Grz}$-equivalent to
      $\Box p$. Consider the tree skeletons $S_0 \defeq \{ \langle\rangle \}$ and
      $S_1 \defeq \{ \langle\rangle, \langle 0\rangle \}$. Let
      $\mathcal{F}_0 \defeq (S_0^\circ, \leq)$, $\mathcal{F}_1 \defeq (S_1^\circ, \leq)$, and
      $w_j \defeq \langle(0, j)\rangle$ for $j = 0, 1$. Since
      $\varphi \sim_{\mathcal{D}^\dagger_\mathrm{Grz}} \Box p$,
      \begin{itemize}
        \item $\varphi_{\mathcal{F}_1} \{ \langle\rangle \} = (\Box p)_{\mathcal{F}_1} \{
              \langle\rangle \} = \emptyset$;
        \item $\varphi_{\mathcal{F}_1} \{ w_j \} = (\Box p)_{\mathcal{F}_1} \{ w_j \} = \{ w_j \}$
              for $j = 0, 1$;
        \item $\varphi_{\mathcal{F}_0} \{ \langle\rangle \} = (\Box p)_{\mathcal{F}_0} \{
              \langle\rangle \} = \{ \langle\rangle \}$.
      \end{itemize}
      Notice that $f : x \mapsto \langle\rangle$ is a p-morphism from $\mathcal{F}_1$ onto
      $\mathcal{F}_0$, whence
      \begin{gather*}
        \varphi_{\mathcal{F}_1} \{\langle\rangle, w_0, w_1\} = \varphi_{\mathcal{F}_1} f^{-1} \{
        \langle\rangle \} = f^{-1}\varphi_{\mathcal{F}_0} \{ \langle\rangle \} = f^{-1} \{
        \langle\rangle \} = \{\langle\rangle, w_0, w_1\}.
      \end{gather*}
      At the same time,
      $
        \varphi_{\mathcal{F}_1} \{ \langle\rangle \} \cup \varphi_{\mathcal{F}_1} \{ w_0 \} \cup
        \varphi_{\mathcal{F}_1} \{ w_1 \} = \{ w_0, w_1 \}
      $.
      Thus, $\varphi$ is not $\mathrm{Grz}$-additive.
    \end{proof}
    \begin{theorem}
      \label{t:Grz}
      There are exactly six $\mathrm{Grz}$-additive formulas without parameters up to
      $\mathrm{Grz}$-equivalence: $\bot$, $p$, $\Dmd p$, $\top$, $\Dmd\Box p$, and
      $\Dmd\Box p \vee p$. Five of them (all except $\top$) are normal.
    \end{theorem}
    \begin{remark}
        \label{r:Grz-par}
        It is easy to check that there are infinitely many pairwise non-$\rm Grz$-equivalent formulas in ${\rm Fm}(r)$,
        whence there are infinitely many non-equivalent $\delta$-formulas with one parameter in $\rm Grz$ and, 
        even more so, in $\rm S4$.
    \end{remark}
  \subsection{Additive formulas in S4}
    For non-empty sets $T$ and $\Sigma$, consider the following relation on $\Sigma^* \times T$:
    \begin{gather*}
      (\vec a, k) \preceq (\vec b, l) \Leftrightarrow \vec a \leq \vec b.
    \end{gather*}
    It is well known that $\mathrm{S4}$ is the logic of the following class of frames:
    \begin{gather*}
      \mathcal{C}_\mathrm{S4} \defeq \{ (S \times T, \preceq) \mid S \subseteq \omega^* \text{ is a
      finite tree skeleton}, T \subseteq \omega \text{ is a finite non-empty set} \}.
    \end{gather*}
    For $\mathcal{F} = (S \times T, \preceq) \in \mathcal{C}_\mathrm{S4}$, we put
    $\mathcal{F}^\circ \defeq (S^\circ \times T^\circ, \preceq)$, where $T^\circ \defeq T \times B$.
    Clearly, $\mathcal{F}^\circ \vDash \mathrm{S4}$ and
    \[
        \rho : S^\circ\times T^\circ\to S\times T,\quad (v,(c, j))\mapsto(\pi v,c)
    \]
    is a p-morphism $\mathcal{F}^\circ \twoheadrightarrow \mathcal{F}$, whence
    $\mathcal{D}_\mathrm{S4} \defeq \rho^{-1}\mathcal{C}_\mathrm{S4}$ is a good 
    class of general frames for $\mathrm{S4}$.
    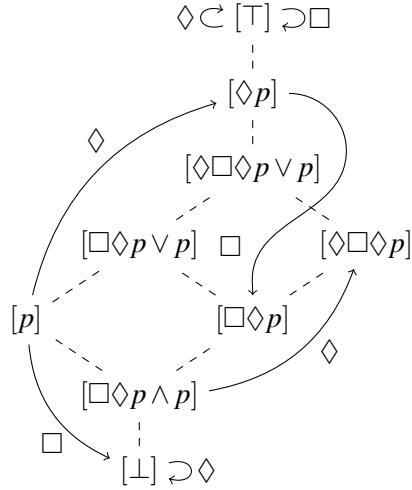
\begin{figure}
        \centering
        \begin{tikzpicture}[
  ord/.style={dashed},
  dx/.style={shift={(#1,0)}},
  dx/.value required,
  dy/.style={shift={(0,#1)}},
  dy/.value required,
  above/.style={dy=1},
  left/.style={dx=-1.5},
  right/.style={dx=1.5},
]


\newcommand{\C}[1]{[#1]}

\node (1) at (0,0) {$\C{\bot}$};
\node (2) at ([above]1) {$\C{\Box\Dmd p\land p}$};
\node (3a) at ([left,above]2) {$\C{p}$};
\node (3b) at ([right,above]2) {$\C{\Box\Dmd p}$};
\node (4a) at ([left,above]3b) {$\C{\Box\Dmd p\lor p}$};
\node (4b) at ([right,above]3b) {$\C{\Dmd\Box\Dmd p}$};
\node (5) at ([left,above]4b) {$\C{\Dmd\Box\Dmd p\lor p}$};
\node (6) at ([above]5) {$\C{\Dmd p}$};
\node (7) at ([above]6) {$\C{\top}$};

\draw[ord] (1) -- (2) -- (3a) -- (4a) -- (5) -- (6) -- (7);
\draw[ord] (2) -- (3b) -- (4a);
\draw[ord] (3b) -- (4b) -- (5);

\draw[loop right,looseness=5,->] (1) to node[dx=.2] {$\Dmd$} (1);
\draw[loop left, looseness=5,->] (7) to node[dx=-.2] {$\Dmd$} (7);
\draw[loop right,looseness=5,->] (7) to node[dx=.2] {$\Box$} (7);

\draw[bend right,looseness=1,->] (3a) to node[label=-92:$\Box$]{} (1);
\draw[bend right,looseness=1,->] (2)  to node[label=0:$\Dmd$]{} (4b);
\draw[bend left, looseness=1,->] (3a) to node[label=$\Dmd$]{} (6);
\draw[->] (6.east) to[in angle=90,curve through={([dx=.15]5.east) ([dy=.3]3b.north)}] (3b.north);
\node[label=93:$\Box$] at ([dy=.25]3b.north) {};

\end{tikzpicture} 
        \caption{The lattice $\mathfrak{Fm}^+(p)/{\sim}_{\mathcal{D}^\dagger_\mathrm{S4}}$} 
        \label{fig:S4}
    \end{figure}
    \begin{proposition}
      The modal lattice $\mathfrak{Fm}(p)/{\sim}_{\mathcal{D}^\dagger_\mathrm{S4}}$ consists of the
      equivalence classes of the formulas from the set
      \begin{gather*}
        \Psi_\mathrm{S4} \defeq \{ \bot,\; p,\; \top,\; \Dmd p,\; \Box\Dmd p,\; \Box\Dmd p \wedge p,\;
        \Box\Dmd p \vee p,\; \Dmd\Box\Dmd p,\; \Dmd\Box\Dmd p \vee p \},
      \end{gather*}
      where operations are shown in Figure~\ref{fig:S4} (missing arrows for $\Dmd$ and $\Box$
      can be restored by Lemma~\ref{S4-lat}).
    \end{proposition}
    \begin{proof}
      Similarly to the case of $\mathrm{Grz}$, it is easy to check that all formulas from 
      $\Psi_\mathrm{S4}$ are pairwise non-equivalent in $\mathcal{D}^\dagger_\mathrm{S4}$ and
      the ordering in Figure~\ref{fig:S4} corresponds to the ordering in $\mathrm{S4}$. 
      Most of the joins and meets are also trivial. 
      We only need to check that $[\Box\Dmd p \vee p] \wedge [\Dmd\Box\Dmd p] = [\Box\Dmd p]$
      or, more precisely, that $\mathcal{D}^\dagger_\mathrm{S4}\vDash p \wedge \Dmd\Box\Dmd p \to \Box\Dmd p$.
      Suppose that $\mathcal{M} = (S^\circ \times T^\circ, \preceq, \vartheta) \in \mathcal{D}^\dagger_\mathrm{S4}$,
      $w \in \vartheta(p)$ and $\mathcal{M}, w \vDash \Dmd\Box\Dmd p$. Then there is
      $v \in S^\circ \times T^\circ$ such that $w \preceq v$ and $\mathcal{M}, v \vDash \Box\Dmd p$.
      In particular, $\mathcal{M}, v \vDash \Dmd p$. Since $|\vartheta(p)| \leq 1$, $v \preceq w$.
      Therefore, $w \preceq u \Leftrightarrow v \preceq u$ for all $u$, whence
      $\mathcal{M}, w \vDash \Box\Dmd p$.

      Now, let us verify that the modal operators in Figure~\ref{fig:S4} correspond to those in
      $\mathfrak{Fm}(p)/{\sim}_{\mathcal{D}^\dagger_\mathrm{S4}}$. 
      The only non-trivial implications here are $\Dmd\Box\Dmd p \to \Dmd(\Box\Dmd p \wedge p)$ and
      $\Box p \to \bot$. The first is derivable in $\mathrm{S4}$. Let us check that
      $\mathcal{D}^\dagger_\mathrm{S4} \vDash \neg\Box p$. Indeed, for each
      $w = (w', (c, j)) \in S^\circ \times T^\circ$,
      $v \defeq (w', (c, 1-j)) \in S^\circ \times T^\circ$ and $w \preceq v$. 
      Since the worlds $w$ and $v$ are distinct, $p$ is false in at least one of them in every 
      $\mathcal{M}\in\mathcal{D}^\dagger_\mathrm{S4}$, whence $\mathcal{M},w\nvDash\Box p$.
    \end{proof}
    \begin{proposition}
      No $\mathrm{S4}$-additive formula is $\mathcal{D}^\dagger_\mathrm{S4}$-equivalent to any of
      the following formulas: $\Box\Dmd p$, $\Box\Dmd p \wedge p$, and $\Box\Dmd p \vee p$.
    \end{proposition}
    \begin{proof}
      Let $S_0$, $S_1$, $w_0$, and $w_1$ be as in the proof of Proposition~\ref{p:Grz-nadd},
      $T \defeq \{0\}$, $c_j \defeq (0, j)$ for $j = 0, 1$. Consider the frames
      $\widehat{\mathcal F}_0 \defeq (S_0^\circ \times T^\circ, \preceq)$ and
      $\widehat{\mathcal F}_1 \defeq (S_1^\circ \times T^\circ, \preceq)$,. Suppose that
      $\varphi \sim_{\mathcal{D}^\dagger_\mathrm{S4}} \psi$, where
      $\psi \in \{ \Box\Dmd p, \Box\Dmd p \wedge p, \Box\Dmd p \vee p \}$. 
      One can check that operators $\varphi_{\widehat{\mathcal F}_0}$ and
      $\varphi_{\widehat{\mathcal F}_1}$ act on singletons as follows: 
      \[\begin{array}{c|c|c|c}
          \psi & \varphi_{\widehat{\mathcal F}_0}\{(\langle\rangle, c_1)\}  &
          \varphi_{\widehat{\mathcal F}_1}\{(\langle\rangle, c_1)\} &
          \varphi_{\widehat{\mathcal F}_1}\{(w_j, c_1)\} \rule[-1.2ex]{0pt}{0pt}
        \\ \hline
        \Box\Dmd p
        & \{(\langle\rangle,c_0),(\langle\rangle,c_1)\}  & \emptyset & \{(w_j,c_0),(w_j,c_1)\} \rule{0pt}{2.2ex}\\
        \Box\Dmd p \wedge p
        & \{(\langle\rangle, c_1)\}  & \emptyset & \{(w_j, c_1)\}\\
        \Box\Dmd p \vee p
        & \{(\langle\rangle,c_0),(\langle\rangle,c_1)\} & \{(\langle\rangle, c_1)\} & \{(w_j,c_0),(w_j,c_1)\}
      \end{array}\]
      Since $\hat f : (x, c_j) \mapsto (\langle\rangle, c_j)$ is a p-morphism from
      $\widehat{\mathcal F}_1$ onto $\widehat{\mathcal F}_0$,
      \begin{gather*}
        \varphi_{\widehat{\mathcal F}_1}\{(\langle\rangle,c_1),(w_0,c_1),(w_1,c_1)\} =
        \varphi_{\widehat{\mathcal F}_1}\hat f^{-1}\{(\langle\rangle, c_1)\} =
        \hat f^{-1}\varphi_{\widehat{\mathcal F}_0}\{(\langle\rangle, c_1)\}.
      \end{gather*}
      It is easy to check that, in all three cases,
      \begin{gather*}
        \hat f^{-1}\varphi_{\widehat{\mathcal F}_0}\{(\langle\rangle, c_1)\} \neq
        \varphi_{\widehat{\mathcal F}_1}\{(\langle\rangle, c_1)\} \cup
        \varphi_{\widehat{\mathcal F}_1}\{(w_0, c_1)\} \cup \varphi_{\widehat{\mathcal F}_1}\{(w_1,
        c_1)\},
      \end{gather*}
      whence $\varphi$ is not $\mathrm{S4}$-additive.
    \end{proof}
    \begin{theorem}
      \label{t:S4}
      There are exactly six $\mathrm{S4}$-additive formulas without parameters up to
      $\mathrm{S4}$-equivalence: $\bot$, $p$, $\Dmd p$, $\top$, $\Dmd\Box\Dmd p$, and
      $\Dmd\Box\Dmd p \vee p$. Five of them (all except $\top$) are normal.
    \end{theorem}
\section{Some notes on interpretations}
  \label{s:interp}
  Let us make a few remarks regarding the interpretations of normal logics, 
  partly building on the results of this paper.
  This task is not as straightforward as it might seem at first glance, for several reasons:
  \begin{itemize}
      \item Non-equivalent formulas can define interpretations of the same logic. 
      \begin{example}
          In almost all standard modal logics ($\rm K$, $\rm K4$, $\rm GL$, $\rm S4$, $\rm S5$, $\rm K4.3$, etc.),
          $\tau_{\Diamond^np}^{-1}\Lambda=\Lambda$ for all $n\geq1$.
          Such logics are called \emph{iterative} in~\cite{Aba86}.
          At the same time, the formulas $\Dmd^np$ are pairwise non-equivalent in most of them (except the ones extending $\rm S4$).
      \end{example}
      \item Even for fixed $\Lambda$ and $\alpha$, it can be challenging to provide 
      an axiomatization of the logic interpreted in $\Lambda$ by $\tau_\alpha$.
      \begin{example}
          Take $\alpha:=\Dmd^2 p \lor \Dmd^3 p$ and consider the logic $\tau_\alpha^{-1}{\rm K}$. 
          This logic is strictly stronger than $\mathrm{K}$:
          the formula $\Box^2 p \land \Dmd^3 \top \to \Dmd^3 p$ is derivable in it, but not in $\mathrm{K}$.
          The complete axiomatization of $\tau_\alpha^{-1}{\rm K}$ is unknown
          (though it follows from Proposition~\ref{p:interp-sem} below that $\tau^{-1}_\alpha{\rm K}$ 
          is the logic of the class of frames $\{(W, R\cup R^2)\mid(W,R)\in\mathcal{C}_{\rm K}\}$). 
      \end{example}
  \end{itemize}
  Therefore, we postpone a systematic study of the interpretability of normal logics 
  to a subsequent paper. 
  Here we will present only the simplest observations. 
  Firstly, one can refine \cite[Theorem~4.27]{Zol00} by proving that 
  infinitely many distinct \emph{normal} logics are interpretable
  in every logic between $\rm K$ and $\rm GL$:
  \begin{proposition}
      \label{p:K-GL-inf}
      Let $\Lambda$ be a normal modal logic such that ${\rm K}\subseteq\Lambda\subseteq{\rm GL.3}$. 
      Then infinitely many distinct normal modal logics are interpretable in $\Lambda$ by modal-to-modal translations
      without parameters.
  \end{proposition}
  \begin{proof}
    Let $\alpha_n:=\Diamond p\land\Box^{n+1}\bot$ for $n\in\omega$.
    Clearly, $\alpha_n$ are normal $\Lambda$-additive, 
    whence the logics $\tau_{\alpha_n}^{-1}\Lambda$ are normal.
    It remains to show that these logics do not coincide for distinct $n\in\omega$.

    By induction on $k\in\omega$, one can verify that
    $\tau_{\alpha_n}\Diamond^k\top\sim_{\rm K}\Diamond^k\top\land\Box^{n+1}\bot$.
    Then notice that $\Diamond^k\top\land\Box^{n+1}\bot\sim_{\rm K}\bot$ for $k>n$
    and $\Diamond^k\top\land\Box^{n+1}\bot\nsim_{\rm GL.3}\bot$ for $k\leq n$.
    Thus, $\tau_{\alpha_n}^{-1}\Lambda\vdash\Box^k\bot$ iff $k>n$.
  \end{proof}
  The following properties of modal-to-modal translations are straightforward:
  \begin{lemma}
      \label{l:interp-triv}
      For every logic $\Lambda$ and formulas $\alpha, \beta\in{\rm Fm}(p,\vec r)$, $\varphi\in{\rm Fm}$,
      \begin{enumerate}
          \item $\tau_\alpha\tau_\beta\varphi=\tau_{\tau_\alpha\beta}\varphi$,
          whence $\tau_\beta^{-1}\tau_\alpha^{-1}\Lambda=\tau_{\tau_\alpha\beta}^{-1}\Lambda$;
          \item $\tau_{\Diamond p}\varphi=\varphi$, whence $\tau_{\Diamond p}^{-1}\Lambda=\Lambda$;
          \item $\tau^{-1}_p\Lambda={\rm Triv}$ and $\tau^{-1}_\bot\Lambda={\rm Ver}$.
      \end{enumerate}
  \end{lemma}
  From Corollary~\ref{c:S5-add} and Lemma~\ref{l:interp-triv}, we immediately obtain
  \begin{proposition}
      \label{p:S5-interp}
      The following three normal logics, and only they, are interpretable in $\rm S5$ without parameters:
      $\rm Ver$, $\rm Triv$, and $\rm S5$.
  \end{proposition}
  For what follows, we need to describe the semantics of the interpreted logics in terms of that of the interpreting logic.
  It is convenient to work with modal algebras for this purpose.
  Basic definitions and facts can be found, for example, in \cite[Section~7]{ChZa97}. 
  Note, however, that in the present paper we treat a modal algebra $\mathfrak{A}$ 
  as a Boolean algebra $\mathfrak{A}^\circ$ equipped with an operator $\Diamond$ rather than $\Box$.
  Also, $\Dmd$ can be an arbitrary operator, possibly not normal and additive;
  whence ${\rm Log\,}\mathfrak{A}$ can be an arbitrary congruential logic.
  For every formula $\alpha(p, \vec r)$, consider the operator $\alpha_\mathfrak{A}:A^{1+n}\to A$
  defined in the same way as in the case of Kripke frames 
  (namely, $\alpha_\mathfrak{A}(U,\vec V)=\vartheta(\alpha)$ whenever $\vartheta(p)=U$ and $\vartheta(\vec r)=\vec V$).
  Let $\sigma_\alpha\mathfrak{A}$ denote the class of algebras
  $\{(\mathfrak{A}^\circ, \alpha_\mathfrak{A}(\cdot, \vec V))\mid \vec V\in A^n\}$.
  It is easy to verify that $\tau_\alpha^{-1}{\rm Log\,}\mathfrak{A}={\rm Log\,}\sigma_\alpha\mathfrak{A}$.
  Since every logic is complete with respect to algebraic semantics, $\sigma_\alpha$ provides an equivalent
  view of modal-to-modal interpretations.

  Recall that each Kripke frame $\mathcal{F}=(W, R)$ corresponds to the algebra $\mathfrak{A}=(\mathfrak{P}(W), R^{-1})$,
  where $\mathfrak{P}(W)$ is the Boolean algebra of subsets of $W$, 
  in the sense that valuations (and hence all semantical notions) on $\mathcal{F}$ and $\mathfrak{A}$ coincide.
  Conversely, if an algebra $\mathfrak{A}$ is of the form $(\mathfrak{P}(W), \Diamond)$ with $\Diamond$ completely additive, 
  then the frame $\mathcal{F}=(W, R_\Diamond)$, where $w\mathrel{R_{\Diamond}} v:\Leftrightarrow w \in\Diamond\{v\}$, 
  corresponds to $\mathfrak{A}$ in the same sense.

  For interpretations, this correspondence yields the following:
  if $\alpha$ is completely additive in a Kripke frame $\mathcal{F}=(W, R)$, 
  then the class of algebras $\sigma_\alpha(\mathfrak{P}(W),R^{-1})$ corresponds to the class of Kripke frames
  $\sigma_\alpha\mathcal{F} := \{(W, R_{\alpha_\mathfrak{A}(\cdot, \vec V)})\mid \vec V\in A^n\}$.
  Therefore, $\tau_\alpha^{-1}{\rm Log\,}\mathcal{F}={\rm Log\,}\sigma_\alpha\mathcal{F}$.
  Similarly, if $\alpha$ is completely additive in a class of Kripke frames $\mathcal{C}$,
  then $\tau_\alpha^{-1}{\rm Log\,}\mathcal{C}={\rm Log\,}\sigma_\alpha\mathcal{C}$,
  where $\sigma_\alpha\mathcal{C}=\bigcup_{\mathcal{F}\in\mathcal{C}}\sigma_\alpha\mathcal{F}$.
  This immediately gives us the following:
  \begin{proposition}
      \label{p:interp-sem}
      Suppose that a normal modal logic $\Lambda$ is interpretable in a normal modal logic $\Lambda_0$ with the finite model property.
      Then $\Lambda$ has the finite model property, namely $\Lambda={\rm Log\,}\sigma_\alpha\mathcal{C}$,
      where $\alpha$ is a formula defining interpretation and $\mathcal{C}$ is a class of finite Kripke frames 
      such that ${\rm Log\,}\mathcal{C}=\Lambda_0$.
  \end{proposition}
  For a tree skeleton $S$, a non-empty set $T$, $v, u\in S$, and $i, j\in T$, we put 
  $v\leq'_S u$ iff $v\leq u$ and $u$ is maximal in $(S, \leq)$ and
  $(v, i)\preceq'_S(u,j)\Leftrightarrow v \leq'_S u$
  (compare with $R'$ from the proof of Proposition~\ref{p:s4-dbd}).
  \begin{proposition}
      The following five normal logics, and only they, are interpretable in $\rm Grz$ without parameters:
      $\rm Ver$, $\rm Triv$, $\rm Grz$, ${\rm Log\,}(S_2,{\leq'_{S_2}})$, and ${\rm Log\,}(S_2,{\leq})$,
      where $S_2 = \{\langle\rangle\}\cup\{\langle k\rangle\mid k\in\omega\}$.
  \end{proposition}
  \begin{proof}
      In view of Theorem~\ref{t:Grz}, it is sufficient to check that these five logics are interpreted in $\rm Grz$ 
      by the formulas $\bot$, $p$, $\Diamond p$, $\Diamond\Box p$, and $\Diamond\Box p\lor p$ respectively.
      The first three interpretations follow from Lemma~\ref{l:interp-triv}.

      Let $\Lambda:=\tau^{-1}_\alpha{\rm Grz}$, where $\alpha:=\Diamond\Box p$, $\mathcal{F}_2:= (S_2,\leq'_{S_2})$.
      We need to show that $\Lambda={\rm Log\,}\mathcal{F}_2$.
      Notice that $\Diamond\Box p\sim_{\rm Grz}\Diamond\Box\Diamond p$, whence, 
      according to the proof of Proposition~\ref{p:s4-dbd}, 
      for a $\rm Grz$-frame $\mathcal{F}=(S,\leq)$,
      $\alpha_\mathcal{F}$ is completely additive
      and $\sigma_\alpha\mathcal{F}=\{(S,\leq'_S)\}$.
      In particular, $\sigma_\alpha(S_2,\leq)=\{\mathcal{F}_2\}$, whence $\mathcal{F}_2\vDash\Lambda$.
      For the converse, we will use the following fact, which can be easily derived from the basic properties of
      generated frames and p-morphisms (cf.~\cite[Corollary~3.16]{ChZa97}):
      \begin{lemma}
            \label{l:red}
          Let $\mathcal{C}$ be a class of Kripke frames and $\mathcal{F}_0=(W_0,R_0)$ be a Kripke frame.
          Suppose that, for every $\mathcal{F}=(W,R)\in\mathcal{C}$ and every world $w\in W$
          there is a world $w_0\in W_0$ with a p-morphism 
          from $\mathcal{F}_0\uparrow w_0$ onto $\mathcal{F}\uparrow w$,
          where $\mathcal{F}\uparrow w$ denotes the subframe of $\mathcal{F}$ generated by $w$.
          Then ${\rm Log\,}\mathcal{F}_0\subseteq{\rm Log\,}\mathcal{C}$.
      \end{lemma}
      By Proposition~\ref{p:interp-sem}, $\Lambda={\rm Log\,}\sigma_{\alpha}\mathcal{C}_{\rm Grz}$.
      Let $\mathcal{F}=(S,\leq'_S)\in\sigma_\alpha\mathcal{C}_{\rm Grz}$, $w\in S$.
      If $w$ is maximal, then $\mathcal{F}\uparrow w$ contains only one reflexive point and is isomorphic
      to $\mathcal{F}_2\uparrow\langle0\rangle$.
      Otherwise, $\mathcal{F}\uparrow w$ consists of $w$ and some maximal points $u_0, \dots,u_{n-1}$, $n\geq 1$.
      Then the mapping $\langle\rangle\mapsto w$, $\langle k\rangle\mapsto u_{k {\,\rm mod\,} n}$ is a p-morphism
      from $\mathcal{F}_2$ onto $\mathcal{F}\uparrow w$.
      Thus, by Lemma~\ref{l:red}, ${\rm Log\,}\mathcal{F}_2\subseteq\Lambda$ and
      $\tau_{\Diamond\Box p}^{-1}{\rm Grz}={\rm Log\,}(S,\leq'_S)$.

      For the last interpretation, it suffices to notice that $\sigma_{\Diamond p\lor p}(S_2,\leq'_{S_2})=(S_2,\leq)$, whence
      \[
        {\rm Log\,}(S_2,\leq)=\tau_{\Diamond p\lor p}^{-1}{\rm Log\,}(S_2,\leq'_{S_2})=\tau^{-1}_{\Diamond p\lor p}\tau^{-1}_{\Diamond\Box p}{\rm Grz}=\tau^{-1}_{\Diamond\Box p\lor p}{\rm Grz}
      \]
      by the first item of Lemma~\ref{l:interp-triv}.
  \end{proof}
  It is easy to adopt the above argument to the case of $\rm S4$ and obtain the following characterization:
  \begin{proposition}
      The following five normal logics, and only they, are interpretable in $\rm S4$ without parameters:
      $\rm Ver$, $\rm Triv$, $\rm Grz$, ${\rm Log\,}(\hat S_2,{\preceq'_{S_2}})$, and ${\rm Log\,}(\hat S_2,{\preceq})$,
      where $\hat S_2:=\bigl\{(\langle\rangle, 0)\bigr\}\cup\bigl\{(\langle k\rangle, i)\bigm| k,i\in\omega\bigr\}\subset S_2\times\omega$.
  \end{proposition}
   Now consider interpretations with parameters. 
   \begin{example}
      \label{e:wGrz-in-GL}
      Let $\alpha:=\Diamond p\lor (r\land p)$. It is easy to see that, for a Kripke frame $\mathcal{F}=(W,R)$,
      \[
        \sigma_\alpha\mathcal{F}=\{(W, R\cup{\rm id}_V)\mid V\subseteq W\},\quad\text{where }{\rm id}_V:=\{(v, v)\mid v\in V\}.
      \]
      At the same time, it is well known that $\rm wGrz$-frames can be obtained from $\rm GL$-frames precisely by
      this procedure. Thus, by Proposition~\ref{p:interp-sem}, ${\rm wGrz}=\tau_\alpha^{-1}{\rm GL}$.
  \end{example}
  Using the characterization of normal $\rm GL$-additive formulas,
  one can show that $\rm wGrz$ is not interpretable in $\rm GL$ without parameters, as
  no disjunction of non-parametric $\delta$-formulas can define such an interpretation.
  While this claim is fairly intuitive, the proof we are aware of is not entirely straightforward.

  In Section~\ref{s:intro}, it was noted  that $\tau_{\Dmd p\lor p}$ defines interpretations of $\rm KT$, $\rm S4$, and $\rm Grz$
  in $\rm K$, $\rm K4$, and $\rm GL$. 
  At the same time, non-parametric interpretations in the reverse direction do not exist,
  since $\rm K$, $\rm K4$, and $\rm GL$ contain non-trivial variable-free formulas while
  $\rm KT$, $\rm S4$, and $\rm Grz$ do not (see~\cite[Lemma~5.1]{Zol00}).
  The situation changes if we allow parameters:
  \begin{example}
      \label{e:K-in-KT}
      \newcommand{\barLe}[1]{\mkern 3.mu\overline{\mkern-3.mu#1\mkern-1.5mu}\mkern 1.5mu}
      \newcommand{\barF}[1]{\mkern 5.5mu\overline{\mkern-5.5mu#1\mkern-1.5mu}\mkern 1.5mu}
      Let $\alpha:=\bigl(\neg r\land\Diamond (r\land p)\bigr)\lor \bigl(r\land\Diamond (\neg r\land p)\bigr)$.
      We will show that ${\rm K}=\tau^{-1}_\alpha\rm KT$.
      Recall that $\rm KT$ is the logic of the class of frames
      $\mathcal{C}_{\rm KT}=\{(S,\barLe\lessdot)\mid(S,\lessdot)\in\mathcal{C}_{\rm K}\}$,
      where $\barLe\lessdot$ is the reflexive closure of $\lessdot$.
      By Proposition~\ref{p:interp-sem}, $\tau_\alpha^{-1}{\rm KT}={\rm Log\,}\sigma_\alpha\mathcal{C}_{\rm KT}$.
      We claim that $\mathcal{C}_{\rm K}\subseteq\sigma_\alpha\mathcal{C}_{\rm KT}$, whence ${\rm K}\supseteq\tau_\alpha^{-1}{\rm KT}$
      (the converse is clear).
      Indeed, every $\mathcal{F}=(S,\lessdot)\in\mathcal{C}_{\rm K}$ can be written as
      $(S,R_{\alpha_{\barF{\mathcal{F}}}(\cdot,V)})$, where 
      $\barF{\mathcal{F}}:=(S,\barLe\lessdot)\in\mathcal{C}_{\rm KT}$ and $V$ consists of all sequences $v\in S$ of even length.
      
      In a similar vein, one can show that $\tau_\beta$, where
      $\beta:=\bigl(\neg r\land\Diamond (r\land\Dmd p)\bigr)\lor \bigl(r\land\Diamond (\neg r\land\Dmd p)\bigr)$,
      defines interpretations of $\rm K4$ in $\rm S4$ and of $\rm GL$ in $\rm Grz$.
  \end{example}
  The role of parameters is also clear in the following example (cf.\ Proposition~\ref{p:S5-interp} and~\cite[Theorem~4.21]{Zol00}).
  \begin{proposition}
      \label{p:S5-par-inf}
      Infinitely many non-equivalent normal logics are interpretable in $\rm S5$ with parameters.
  \end{proposition}
  \begin{proof}
      For $n>0$, $\vec r=\langle r_l\rangle_{l<n}$, consider the formula 
      $\alpha_n(p,\vec r):=\bigvee_{k\leq n} \bigl((\gamma^k_n\lor\gamma^n_n)\land\Diamond(\gamma^k_n\land p)\bigr)$, where
      \[
        \gamma^k_n(\vec r):=\begin{cases}
            r_k\land\bigwedge_{l<n,l\neq k}\neg r_l &\text{for }k<n,\\
            \bigwedge_{l<n}\neg\gamma^l_n &\text{for k = n}.
        \end{cases}
      \]
      Let $\mathcal{F}_{3,n}:=(S_{3,n}, \preceq)$, where 
      $S_{3,n}:=\{\langle\rangle, \langle0\rangle, \dots,\langle n-1\rangle\}\times\omega$.
      It is easy to show that the logics of $\mathcal{F}_{3,n}$ are distinct: the formula 
      $\bigwedge_{k<m}\Diamond\bigl(q_k\land\bigwedge_{l<m,l\neq k}\neg\Diamond q_l\bigr)\to\bot$
      is valid in $\mathcal{F}_{3,n}$ iff $n<m$.
      
      Let us show that $\tau_{\alpha_n}^{-1}{\rm S5}={\rm Log\,}\mathcal{F}_{3,n}$. 
      It is well known that ${\rm S5}={\rm Log\,}\mathcal{F}_4$, where $\mathcal{F}_4:=(\omega,\omega\times\omega)$.
      By Proposition~\ref{p:interp-sem}, $\tau_{\alpha_n}^{-1}{\rm S5}={\rm Log\,}\sigma_{\alpha_n}\mathcal{F}_4$.
      So, it remains to check that ${\rm Log\,}\mathcal{F}_{3,n}={\rm Log\,}\sigma_{\alpha_n}\mathcal{F}_4$.

      Consider the frame $\mathcal{F}=(\omega,R_{\alpha_n(\cdot,\vec V)})\in\sigma_{\alpha_n}\mathcal{F}_4$,
      where $\vec V:=\langle V_k\rangle_{k<n}$, $V_k:=\{i\in\omega\mid i\,{\rm mod}\,(n+1)=k\}$ for $k<n$.
      Let $V_n:=\omega\setminus\bigcup_{k<n}V_k$.
      Notice that, if $\vartheta(\vec r)=\vec V$, then $\vartheta(\gamma_n^k)=V_k$ for $k\leq n$,
      whence 
      \[
        i \mathrel{R_{\alpha_n(\cdot,\vec V)}} j \;\Leftrightarrow\;
        i\in(\alpha_n)_{\mathcal{F}_4}(\{j\},\vec V) \;\Leftrightarrow\;
        \exists{k\leq n}\,(i\in V_k\cup V_n \land j\in V_k).
      \]
      Therefore, $\mathcal{F}\cong\mathcal{F}_{3,n}$, where the isomorphism maps 
      $i$ to $(\langle k\rangle, s)$ if $i=s(n+1)+k$, $s\in\omega$, $k<n$
      and to $(\langle\rangle, s)$ if $i=s(n+1)+n$, $s\in\omega$.
      Hence, ${\rm Log\,}\mathcal{F}_{3,n}\supseteq{\rm Log\,}\sigma_{\alpha_n}\mathcal{F}_4$.

      To prove the converse inclusion, by Lemma~\ref{l:red}, it suffices to show that, 
      for each $\mathcal{F}\in\sigma_{\alpha_n}\mathcal{F}_4$
      and $w\in\omega$, there is $w_0\in S_{3,n}$ such that $\mathcal{F}\uparrow w$ is a p-morphic 
      image of $\mathcal{F}_{3,n}\uparrow w_0$.
      Let $\vec V=\langle V_k\rangle_{k<n}\in\mathcal{P}(\omega)^n$ be such that $\mathcal{F}=(\omega,R_{\alpha_n(\cdot,\vec V)})$,
      $\hat V_k:=V_k\setminus\bigcup_{l<n,l\neq k}V_l$ for $k<n$, and $\hat V_n := \omega\setminus\bigcup_{k<n}\hat V_k$.
      It is easy to see that,
      if $\vartheta(\vec r)=\vec V$, then $\vartheta(\gamma_n^k)=\hat V_k$ for $k\leq n$, whence
      \[
        i \mathrel{R_{\alpha_n(\cdot,\vec V)}} j \;\Leftrightarrow\;
        i\in(\alpha_n)_{\mathcal{F}_4}(\{j\},\vec V) \;\Leftrightarrow\;
        \exists{k\leq n}\,(i\in \hat V_k\cup \hat V_n \land j\in \hat V_k).
      \]
      If $w\in \hat V_k$ for some $k<n$ or $w\in\hat V_n=\omega$, 
      then $\mathcal{F}\uparrow w$ consists of only one cluster 
      (which is $\hat V_k\subseteq\omega$ or $\hat V_n=\omega$)
      and it is a p-morphic image of the infinite cluster $\mathcal{F}_{3,n}\uparrow(\langle 0\rangle,0)$.
      If $w\in \hat V_n\subsetneq\omega$, then $\mathcal{F}\uparrow w=\mathcal{F}$ and $\mathcal{F}$ 
      consists of the root cluster $\hat V_n$,
      which sees at least one and at most $n$ clusters (i.e., non-empty sets among $\hat V_k$, $k<n$).
      One can verify that $\mathcal{F}$ is a p-morphic image of the whole frame $\mathcal{F}_{3,n}$ in this case.
  \end{proof}
\section{Problems for future research}
    \label{s:prob}
  Of course, it would be desirable to obtain a general characterization of additive formulas in
  normal modal logics. However, the results obtained in this paper cast doubt on the possibility of
  such a characterization, even for well-behaved classes of logics (for example, logics corresponding
  to Horn frame conditions or extensions of $\mathrm{S4}$). Therefore, we formulate here only some
  of the most interesting and, at the same time, relatively accessible questions related to the
  characterization of additive formulas in specific logics:
  \begin{enumerate}
    \item Describe all additive formulas with parameters in $\mathrm{S4}$ and $\mathrm{Grz}$, 
          all additive formulas in $\mathrm{K4}$.
    \item Describe all additive formulas in $\mathrm{K4}.3$ and $\mathrm{S4}.3$. Is it true that all
          additive formulas in these logics are equivalent to $p$-positive ones? It is known that
          this is not the case for monotone formulas~\cite[Proposition~4.7]{Dvo26}.
    \item Describe all \emph{biadditive} formulas in $\mathrm{K}$, i.e., formulas $\alpha(p, q)$ such
          that
          \begin{align*}
            \alpha(p \vee p', q) &\sim_\mathrm{K} \alpha(p, q) \vee \alpha(p', q),\\
            \alpha(p, q \vee q') &\sim_\mathrm{K} \alpha(p, q) \vee \alpha(p, q').
          \end{align*}
  \end{enumerate}
  We also pose several questions regarding interpretations of modal logics:
  \begin{enumerate}
      \setcounter{enumi}{3}
      \item How many normal logics are interpretable with parameters in ${\rm S4}$ and $\rm Grz$?
            It seems plausible that there are infinitely many of them in every logic between $\rm K$ and $\rm Grz$ 
            (cf.\ Remark~\ref{r:Grz-par}) and possibly even in every non-tabular normal logic 
            (cf.\ Proposition~\ref{p:S5-par-inf}).
      \item Is it true that every normal logic interpretable in $\mathrm{K}$ ($\mathrm{GL}$,
            $\mathrm{S4}$, etc.) is finitely axiomatizable? If so, what is the complexity of computing
            such an axiomatization from a formula defining the interpretation?
  \end{enumerate}
  As noted after Example~\ref{e:wGrz-in-GL}, the characterization of additive formulas makes 
  it possible to prove the absence of interpretability of a certain modal logic in a given logic 
  by checking all additive formulas. 
  However, this approach is cumbersome. At the same time, it seems quite obvious that, 
  for example, the logics K and K4 are not interpretable in each other, even with parameters. 
  In this regard, it makes sense to try to develop general necessary conditions for interpretability. 
  For example, one might suppose that transitive and non-transitive logics cannot be interpreted in each other.
  However, this claim is obviously wrong. 
  To state a more accurate conjecture, we need an additional notion. 
  A normal modal logic $\Lambda$ is \emph{pretransitive} if $\Lambda\vdash\Diamond^{n+1}p\to\bigvee_{k\leq n}\Diamond^kp$
  for some $n\geq1$.
  \begin{conjecture}
      \label{conj}
      If a normal modal logic $\Lambda$ is interpretable by a modal-to-modal translation (with parameters) 
      in a pretransitive normal logic, then $\Lambda$ is pretransitive.
  \end{conjecture}
  Notice that this conjecture, if it is true, imposes limitations on possible weaker-to-stronger translations 
  discussed in~\cite{Hum06}.
  \begin{enumerate}
      \setcounter{enumi}{5}
      \item Does Conjecture~\ref{conj} hold? What can be said about interpretations of pretransitive logics in 
      ``anti-transitive'' logics like $\rm K$, $\rm KB$, $\rm KT$? 
      What are other necessary conditions for a logic $\Lambda$ to be interpretable in a logic ${\rm \Lambda}_0$
      with parameters?
      Notice that some necessary conditions for interpretations without parameters are given in~\cite[Section~5]{Zol00}.
  \end{enumerate}
\section*{Acknowledgments}
  I would like to thank my supervisor L.~D.~Beklemishev for stimulating discussions and continuous support.
  I am also grateful to the anonymous reviewers for their numerous helpful remarks, which significantly improved the paper.
\bibliographystyle{eptcs}
\bibliography{generic}

\begin{thebibliography}{10}
\providecommand{\bibitemdeclare}[2]{}
\providecommand{\surnamestart}{}
\providecommand{\surnameend}{}
\providecommand{\urlprefix}{Available at }
\providecommand{\url}[1]{\texttt{#1}}
\providecommand{\href}[2]{\texttt{#2}}
\providecommand{\urlalt}[2]{\href{#1}{#2}}
\providecommand{\doi}[1]{doi:\urlalt{https://doi.org/#1}{#1}}
\providecommand{\eprint}[1]{arXiv:\urlalt{https://arxiv.org/abs/#1}{#1}}
\providecommand{\bibinfo}[2]{#2}

\bibitemdeclare{inproceedings}{Aba86}
\bibitem{Aba86}
\bibinfo{author}{Merab~A. \surnamestart Abashidze\surnameend}
  (\bibinfo{year}{1986}): \emph{\bibinfo{title}{The Iterativity Property in
  Provability Logics}}.
\newblock In: {\slshape \bibinfo{booktitle}{Abstracts of the 8th All-Union
  Conference on Mathematical Logic}}, \bibinfo{address}{Moscow},
  p.~\bibinfo{pages}{4}.
\newblock \bibinfo{note}{In Russian}.

\bibitemdeclare{article}{Benth98}
\bibitem{Benth98}
\bibinfo{author}{Johan \surnamestart van Benthem\surnameend}
  (\bibinfo{year}{1998}): \emph{\bibinfo{title}{Program constructions that are
  safe for bisimulation}}.
\newblock {\slshape \bibinfo{journal}{Studia Logica}}
  \bibinfo{volume}{60}(\bibinfo{number}{2}), pp. \bibinfo{pages}{311--330},
  \doi{10.1023/A:1005072201319}.

\bibitemdeclare{misc}{Benth26}
\bibitem{Benth26}
\bibinfo{author}{Johan \surnamestart van Benthem\surnameend},
  \bibinfo{author}{Balder \surnamestart ten Cate\surnameend} \&
  \bibinfo{author}{Xi~\surnamestart Yang\surnameend} (\bibinfo{year}{2026}):
  \emph{\bibinfo{title}{When do modal definability and preservation theorems
  transfer to the finite?}}, \doi{10.48550/arXiv.2603.12171}.
\newblock \eprint{2603.12171}.

\bibitemdeclare{misc}{BCI25}
\bibitem{BCI25}
\bibinfo{author}{Nick \surnamestart Bezhanishvili\surnameend},
  \bibinfo{author}{Balder \surnamestart ten Cate\surnameend} \&
  \bibinfo{author}{Rosalie \surnamestart Iemhoff\surnameend}
  (\bibinfo{year}{2025}): \emph{\bibinfo{title}{Six proofs of interpolation for
  the modal logic {K}}}, \doi{10.48550/arXiv.2510.16398}.
\newblock \eprint{2510.16398}.

\bibitemdeclare{book}{ChZa97}
\bibitem{ChZa97}
\bibinfo{author}{Alexander \surnamestart Chagrov\surnameend} \&
  \bibinfo{author}{Michael \surnamestart Zakharyaschev\surnameend}
  (\bibinfo{year}{1997}): \emph{\bibinfo{title}{Modal Logic}}.
\newblock {\slshape \bibinfo{series}{Oxford Logic
  Guides}}~\bibinfo{volume}{35}, \bibinfo{publisher}{Clarendon Press},
  \bibinfo{address}{Oxford}, \doi{10.1093/oso/9780198537793.001.0001}.

\bibitemdeclare{article}{Dvo24}
\bibitem{Dvo24}
\bibinfo{author}{Lev~V. \surnamestart Dvorkin\surnameend}
  (\bibinfo{year}{2024}): \emph{\bibinfo{title}{On provability logics of
  {Niebergall} arithmetic}}.
\newblock {\slshape \bibinfo{journal}{Izvestiya Rossiiskoi Akademii Nauk.
  Seriya Matematicheskaya}} \bibinfo{volume}{88}(\bibinfo{number}{3}), pp.
  \bibinfo{pages}{468--505}, \doi{10.4213/im9524e}.

\bibitemdeclare{misc}{Dvo26}
\bibitem{Dvo26}
\bibinfo{author}{Lev~V. \surnamestart Dvorkin\surnameend}
  (\bibinfo{year}{2026}): \emph{\bibinfo{title}{Monotonicity versus positivity
  in modal logics}}, \doi{10.48550/arXiv.2602.02837}.
\newblock \eprint{2602.02837}.

\bibitemdeclare{phdthesis}{French10}
\bibitem{French10}
\bibinfo{author}{Rohan \surnamestart French\surnameend} (\bibinfo{year}{2010}):
  \emph{\bibinfo{title}{Translational embeddings in modal logic}}.
\newblock Ph.D. thesis, \bibinfo{school}{Monash University},
  \doi{10.26180/4546354}.

\bibitemdeclare{book}{GabMaks05}
\bibitem{GabMaks05}
\bibinfo{author}{Dov~M. \surnamestart Gabbay\surnameend} \&
  \bibinfo{author}{Larisa~L. \surnamestart Maksimova\surnameend}
  (\bibinfo{year}{2005}): \emph{\bibinfo{title}{Interpolation and definability.
  {M}odal and intuitionistic logics}}.
\newblock {\slshape \bibinfo{series}{Oxford Logic
  Guides}}~\bibinfo{volume}{46}, \bibinfo{publisher}{Clarendon Press},
  \bibinfo{address}{Oxford}, \doi{10.1093/acprof:oso/9780198511748.001.0001}.

\bibitemdeclare{article}{Gliv29}
\bibitem{Gliv29}
\bibinfo{author}{Valery \surnamestart Glivenko\surnameend}
  (\bibinfo{year}{1929}): \emph{\bibinfo{title}{Sur quelques points de la
  logique de {M}. Brouwer}}.
\newblock {\slshape \bibinfo{journal}{Bulletins de la classe des sciences de
  l'Académie royale de Belgique}} \bibinfo{volume}{15}(\bibinfo{number}{5}),
  pp. \bibinfo{pages}{183--188}.

\bibitemdeclare{article}{God33}
\bibitem{God33}
\bibinfo{author}{Kurt \surnamestart G{\"o}del\surnameend}
  (\bibinfo{year}{1933}): \emph{\bibinfo{title}{Eine Interpretation des
  intuitionistischen Aussagenkalküls}}.
\newblock {\slshape \bibinfo{journal}{Ergebnisse eines Mathematischen
  Kolloquiums}} \bibinfo{volume}{4}, pp. \bibinfo{pages}{39--40}.
\newblock \bibinfo{note}{English translation in G{\"o}del, Kurt.
  \emph{Collected works. Volume I: Publications 1929--1936}. Ed. by Solomon
  Feferman et al. New York: Oxford University Press, 1986., pp. 301--303}.

\bibitemdeclare{article}{Gold78}
\bibitem{Gold78}
\bibinfo{author}{Rob \surnamestart Goldblatt\surnameend}
  (\bibinfo{year}{1978}): \emph{\bibinfo{title}{Arithmetical necessity,
  provability and intuitionistic logic}}.
\newblock {\slshape \bibinfo{journal}{Theoria}}
  \bibinfo{volume}{44}(\bibinfo{number}{1}), pp. \bibinfo{pages}{38--46},
  \doi{10.1111/j.1755-2567.1978.tb00831.x}.

\bibitemdeclare{article}{Humb95}
\bibitem{Humb95}
\bibinfo{author}{Lloyd \surnamestart Humberstone\surnameend}
  (\bibinfo{year}{1995}): \emph{\bibinfo{title}{The logic of non-contingency}}.
\newblock {\slshape \bibinfo{journal}{Notre Dame Journal of Formal Logic}}
  \bibinfo{volume}{36}(\bibinfo{number}{2}), pp. \bibinfo{pages}{214--229},
  \doi{10.1305/ndjfl/1040248455}.

\bibitemdeclare{inproceedings}{Hum06}
\bibitem{Hum06}
\bibinfo{author}{Lloyd \surnamestart Humberstone\surnameend}
  (\bibinfo{year}{2006}): \emph{\bibinfo{title}{Weaker-to-Stronger
  Translational Embeddings in Modal Logic}}.
\newblock In \bibinfo{editor}{Guido \surnamestart Governatori\surnameend},
  \bibinfo{editor}{Ian \surnamestart Hodkinson\surnameend} \&
  \bibinfo{editor}{Yde \surnamestart Venema\surnameend}, editors: {\slshape
  \bibinfo{booktitle}{Advances in Modal Logic}}, \bibinfo{volume}{6},
  \bibinfo{publisher}{College Publications}, \bibinfo{address}{London}, pp.
  \bibinfo{pages}{279--297}.
\newblock \urlprefix\url{http://www.aiml.net/volumes/volume6/}.

\bibitemdeclare{article}{JT51}
\bibitem{JT51}
\bibinfo{author}{Bjarni \surnamestart J{\'o}nsson\surnameend} \&
  \bibinfo{author}{Alfred \surnamestart Tarski\surnameend}
  (\bibinfo{year}{1951}): \emph{\bibinfo{title}{Boolean algebras with
  operators. {I}}}.
\newblock {\slshape \bibinfo{journal}{American Journal of Mathematics}}
  \bibinfo{volume}{73}, pp. \bibinfo{pages}{891--939}, \doi{10.2307/2372123}.

\bibitemdeclare{article}{KW99}
\bibitem{KW99}
\bibinfo{author}{Marcus \surnamestart Kracht\surnameend} \&
  \bibinfo{author}{Frank \surnamestart Wolter\surnameend}
  (\bibinfo{year}{1999}): \emph{\bibinfo{title}{Normal monomodal logics can
  simulate all others}}.
\newblock {\slshape \bibinfo{journal}{The Journal of Symbolic Logic}}
  \bibinfo{volume}{64}(\bibinfo{number}{1}), pp. \bibinfo{pages}{99--138},
  \doi{10.2307/2586754}.

\bibitemdeclare{article}{dR*}
\bibitem{dR*}
\bibinfo{author}{Natasha \surnamestart Kurtonina\surnameend} \&
  \bibinfo{author}{Maarten \surnamestart de~Rijke\surnameend}
  (\bibinfo{year}{1997}): \emph{\bibinfo{title}{Simulating without negation}}.
\newblock {\slshape \bibinfo{journal}{Journal of Logic and Computation}}
  \bibinfo{volume}{7}(\bibinfo{number}{4}), pp. \bibinfo{pages}{501--522},
  \doi{10.1093/logcom/7.4.501}.

\bibitemdeclare{incollection}{KuznMur80}
\bibitem{KuznMur80}
\bibinfo{author}{Aleksandr~V. \surnamestart Kuznetsov\surnameend} \&
  \bibinfo{author}{Alexei~Yu. \surnamestart Muravitsky\surnameend}
  (\bibinfo{year}{1980}): \emph{\bibinfo{title}{Provability as Modality}}.
\newblock In: {\slshape \bibinfo{booktitle}{Current Problems of Logic and
  Methodology of Science}}, \bibinfo{publisher}{Naukova Dumka},
  \bibinfo{address}{Kiev}, pp. \bibinfo{pages}{193--230}.
\newblock \bibinfo{note}{In Russian}.

\bibitemdeclare{article}{Lynd59Hom}
\bibitem{Lynd59Hom}
\bibinfo{author}{Roger~C. \surnamestart Lyndon\surnameend}
  (\bibinfo{year}{1959}): \emph{\bibinfo{title}{Properties preserved under
  homomorphism}}.
\newblock {\slshape \bibinfo{journal}{Pacific Journal of Mathematics}}
  \bibinfo{volume}{9}, pp. \bibinfo{pages}{143--154},
  \doi{10.2140/pjm.1959.9.143}.

\bibitemdeclare{article}{Maks80}
\bibitem{Maks80}
\bibinfo{author}{Larisa~L. \surnamestart Maksimova\surnameend}
  (\bibinfo{year}{1980}): \emph{\bibinfo{title}{Interpolation theorems in modal
  logics. Sufficient conditions}}.
\newblock {\slshape \bibinfo{journal}{Algebra and Logic}}
  \bibinfo{volume}{19}(\bibinfo{number}{2}), pp. \bibinfo{pages}{120--132},
  \doi{10.1007/BF01669837}.

\bibitemdeclare{article}{Maks82}
\bibitem{Maks82}
\bibinfo{author}{Larisa~L. \surnamestart Maksimova\surnameend}
  (\bibinfo{year}{1982}): \emph{\bibinfo{title}{The {Lyndon} interpolation
  theorem in modal logics}}.
\newblock {\slshape \bibinfo{journal}{Trudy Inst. Mat. Sib. Otd. AN SSSR}}
  \bibinfo{volume}{2}, pp. \bibinfo{pages}{45--55}.
\newblock \bibinfo{note}{In Russian}.

\bibitemdeclare{article}{Mats55}
\bibitem{Mats55}
\bibinfo{author}{Kazuo \surnamestart Matsumoto\surnameend}
  (\bibinfo{year}{1955}): \emph{\bibinfo{title}{Reduction theorem in {Lewis}'
  sentential calculi}}.
\newblock {\slshape \bibinfo{journal}{Mathematica Japonicae}}
  \bibinfo{volume}{3}, pp. \bibinfo{pages}{133--135}.

\bibitemdeclare{incollection}{Meschi78}
\bibitem{Meschi78}
\bibinfo{author}{Vyacheslav~Yu. \surnamestart Meskhi\surnameend}
  (\bibinfo{year}{1978}): \emph{\bibinfo{title}{Algebraic analysis of modal
  fragments of temporal logics}}.
\newblock In: {\slshape \bibinfo{booktitle}{Logic, Semantics, Methodology}},
  \bibinfo{publisher}{Metsniereba}, \bibinfo{address}{Tbilisi}, pp.
  \bibinfo{pages}{113--124}.
\newblock \bibinfo{note}{In Russian}.

\bibitemdeclare{article}{MR66}
\bibitem{MR66}
\bibinfo{author}{Hugh \surnamestart Montgomery\surnameend} \&
  \bibinfo{author}{Richard \surnamestart Routley\surnameend}
  (\bibinfo{year}{1966}): \emph{\bibinfo{title}{Contingency and non-contingency
  bases for normal modal logics}}.
\newblock {\slshape \bibinfo{journal}{Logique et Analyse. Nouvelle S{\'e}rie}}
  \bibinfo{volume}{9}(\bibinfo{number}{35/36}), pp. \bibinfo{pages}{318--328}.
\newblock \urlprefix\url{http://www.jstor.org/stable/44083397}.

\bibitemdeclare{article}{Mor19}
\bibitem{Mor19}
\bibinfo{author}{Tommaso \surnamestart Moraschini\surnameend}
  (\bibinfo{year}{2019}): \emph{\bibinfo{title}{Varieties of positive modal
  algebras and structural completeness}}.
\newblock {\slshape \bibinfo{journal}{The Review of Symbolic Logic}}
  \bibinfo{volume}{12}(\bibinfo{number}{3}), pp. \bibinfo{pages}{557--588},
  \doi{10.1017/S1755020319000236}.

\bibitemdeclare{article}{Pell84}
\bibitem{Pell84}
\bibinfo{author}{Francis~J. \surnamestart Pelletier\surnameend}
  (\bibinfo{year}{1984}): \emph{\bibinfo{title}{Six Problems in ``Translational
  Equivalence''}}.
\newblock {\slshape \bibinfo{journal}{Logique et Analyse}}
  \bibinfo{volume}{27}(\bibinfo{number}{108}), pp. \bibinfo{pages}{423--434}.
\newblock \urlprefix\url{http://www.jstor.org/stable/44084104}.

\bibitemdeclare{incollection}{PraMal68}
\bibitem{PraMal68}
\bibinfo{author}{Dag \surnamestart Prawitz\surnameend} \&
  \bibinfo{author}{Per-Erik \surnamestart Malmn{\"a}s\surnameend}
  (\bibinfo{year}{1968}): \emph{\bibinfo{title}{A survey of some connections
  between classical, intuitionistic and minimal logic}}.
\newblock In: {\slshape \bibinfo{booktitle}{Studies in Logic and the
  Foundations of Mathematics}}, \bibinfo{volume}{50},
  \bibinfo{publisher}{Elsevier}, pp. \bibinfo{pages}{215--229},
  \doi{10.1016/S0049-237X(08)70527-5}.

\bibitemdeclare{phdthesis}{dR93}
\bibitem{dR93}
\bibinfo{author}{Maarten \surnamestart de~Rijke\surnameend}
  (\bibinfo{year}{1993}): \emph{\bibinfo{title}{Extending Modal Logic}}.
\newblock Ph.D. thesis, \bibinfo{school}{University of Amsterdam}.
\newblock \urlprefix\url{https://eprints.illc.uva.nl/id/eprint/1961/}.

\bibitemdeclare{article}{Sol76}
\bibitem{Sol76}
\bibinfo{author}{Robert~M. \surnamestart Solovay\surnameend}
  (\bibinfo{year}{1976}): \emph{\bibinfo{title}{Provability interpretations of
  modal logic}}.
\newblock {\slshape \bibinfo{journal}{Israel Journal of Mathematics}}
  \bibinfo{volume}{25}, pp. \bibinfo{pages}{287--304},
  \doi{10.1007/BF02757006}.

\bibitemdeclare{article}{Thom75}
\bibitem{Thom75}
\bibinfo{author}{Steven~K. \surnamestart Thomason\surnameend}
  (\bibinfo{year}{1975}): \emph{\bibinfo{title}{Reduction of Tense Logic to
  Modal Logic {II}}}.
\newblock {\slshape \bibinfo{journal}{Theoria}}
  \bibinfo{volume}{41}(\bibinfo{number}{3}), pp. \bibinfo{pages}{154--169},
  \doi{10.1111/j.1755-2567.1975.tb00555.x}.

\bibitemdeclare{article}{Zol00}
\bibitem{Zol00}
\bibinfo{author}{Evgeni~E. \surnamestart Zolin\surnameend}
  (\bibinfo{year}{2000}): \emph{\bibinfo{title}{Embeddings of propositional
  monomodal logics}}.
\newblock {\slshape \bibinfo{journal}{Logic Journal of the {IGPL}}}
  \bibinfo{volume}{8}(\bibinfo{number}{6}), pp. \bibinfo{pages}{861--882},
  \doi{10.1093/jigpal/8.6.861}.

\bibitemdeclare{article}{Zol02}
\bibitem{Zol02}
\bibinfo{author}{Evgeni~E. \surnamestart Zolin\surnameend}
  (\bibinfo{year}{2002}): \emph{\bibinfo{title}{Sequential Reflexive Logics
  with Noncontingency Operator}}.
\newblock {\slshape \bibinfo{journal}{Mathematical Notes}}
  \bibinfo{volume}{72}(\bibinfo{number}{6}), pp. \bibinfo{pages}{784--798},
  \doi{10.1023/A:1021485712270}.

\end{thebibliography}

\end{document}